\documentclass[final,twoside,11pt]{entics}
\pdfoutput=1

\usepackage{enticsmacro}
\usepackage{tabularx}
\usepackage{xcolor}
\usepackage{underscore}
\usepackage{local-amsthm}
\usepackage{local-typography}
\usepackage{local-abbrv}
\usepackage{local-macros}
\usepackage{local-tikz}
\usepackage{mathpartir}
\usepackage{amsmath}
\usepackage{wasysym}

\def\conf{MFPS 2022} 	%%Fill in the acronym for your conference (with year)
\volume{1}			%Fill in the ENTICS volume number here
\def\volu{1}			% and here
\def\papno{12}			%Fill in your paper number here
\def\mydoi{{\normalshape \href{https://doi.org/10.46298/entics.10323}{10.46298/entics.10323}}}
\def\copynames{D. Palombi, J. Sterliing} %% Fill in the first initial and last name of the authors
\def\runtitle{Classifying Topoi in Synthetic Guarded Domain Theory}
\def\lastname{Palombi and Sterling}

\begin{document}

\begin{frontmatter}
  \title{Classifying Topoi in Synthetic Guarded Domain Theory}
  \subtitle{The Universal Property of Multi-clock Guarded Recursion}
  \author{Daniele Palombi\thanksref{dpl}}
  \author{Jonathan Sterling\thanksref{jms}}

  \address[dpl]{Department of Mathematics\\Sapienza University of Rome}
  \address[jms]{Department of Computer Science\\Aarhus University}

  \begin{abstract}
    Several different topoi have played an important role in the development and
    applications of \emph{synthetic guarded domain theory} (SGDT), a new kind of
    synthetic domain theory that abstracts the concept of \emph{guarded recursion}
    frequently employed in the semantics of programming languages. In order to
    unify the accounts of guarded recursion and coinduction, several authors have
    enriched SGDT with multiple ``clocks'' parameterizing different time-streams,
    leading to more complex and difficult to understand topos models.
    Until now these topoi have been understood very concretely \emph{qua}
    categories of presheaves, and the logico-geometrical question of what
    theories these topoi classify has remained open. We show that several
    important topos models of SGDT classify very simple geometric theories, and
    that the passage to various forms of multi-clock guarded recursion can be
    rephrased more compositionally in terms of the lower bagtopos
    construction of Vickers and variations thereon due to Johnstone.
    We contribute to the consolidation of SGDT by isolating the universal
    property of multi-clock guarded recursion as a modular construction that
    applies to any topos model of single-clock guarded recursion.
  \end{abstract}
\end{frontmatter}

\section{Introduction}

\subsection{Synthetic guarded domain theory}

Beginning with Scott and Strachey's groundbreaking investigations in 1969, the
scientific study of programming semantics has been guided by the search for a
topology of computation --- embodied in monoidal closed categories of spaces
called \DefEmph{domains} whose points can be thought of as the values of datatypes
and computer
programs~\cite{scott:1970:outline,scott:1972,scott:1976,scott:1982}. The thesis
of denotational semantics under Scott and Strachey is that the computational
behavior of expressions in a programming language can be studied by
characterizing what values they take when interpreted as continuous functions
between domains; the advantage of denotational semantics over the
direct/operational study of program behavior is that, unlike the latter, it is
compositional and amenable to mathematical methods of reduction and
abstraction.

The need to reason about increasingly complex programming languages has drawn
researchers toward alternative theories of domains based on (complete, bounded,
ultra-) \DefEmph{metric
spaces}~\cite{macqueen-plotkin-sethi:1984,america-rutten:1987,rutten-turi:1992,breugel-warmerdam:1994,birkedal-stovring-thamsborg:2010:domain}.
Metric domain theory has proved instrumental in untangling the circularities of
programming language semantics involving higher-order
store~\cite{birkedal-stovring-thamsborg:2010}; categories of metric spaces come
equipped with a pointed endofunctor $\frac{1}{2}\cdot-$ to \emph{scale} a given
space by half, which in combination with Banach's fixed point theorem can be
used to prove algebraic compactness relative to mixed-variance endofunctors
whose recursive variables are guarded by $\frac{1}{2}\cdot-$.
From this one obtains a relatively simple interpretation of both recursive types and programs.

Ideas from metric domain theory have been later simplified and generalized to
categories of sheaves on frames with well-founded bases and the even simpler
case of presheaves on well-founded posets \cite{bmss:2011}, allowing for a
\emph{synthetic} and topos-theoretic approach to guarded domain theory
along the lines of synthetic domain theory~\cite{hyland:1991} or synthetic
differential geometry~\cite{kock:2006}. The core idea of \DefEmph{synthetic
guarded domain theory} (\textbf{SGDT}) is to work with a topos $\SCat$ that is equipped with an
endofunctor $\Ltr$ called the \DefEmph{later modality}, together with a natural
transformation $\Mor[\Next]{\Idn{\SCat}}{\Ltr}$ and a ``guarded'' fixed point operator
$-^\dag$ ensuring that for each $\Mor[f]{\Ltr A}{A}$, there exists a unique
$\Mor[f^\dag]{\ObjTerm{\SCat}}{A}$ such that $f \circ \Next_A\circ f^\dag =
f^\dag$.  Under mild assumptions (\eg left exactness of $\Ltr$), it can be seen
that $\Ltr$ extends to an endofunctor on the fundamental fibration
$\FibMor{\FFib{\SCat}}{\SCat}$ and therefore gives rise to a true connective
in the internal dependent type theory of $\SCat$.

Synthetic guarded domain theory has been employed as a metalanguage for the
denotational \emph{and} operational semantics of simple programming languages
such as PCF and
FPC~\cite{paviotti-mogelberg-birkedal:2015,mogelberg-paviotti:2016,paviotti:2016};
the models of \opcit can be seen as synthetic versions of Escard\'o's
``analytic'' metric model of PCF~\cite{escardo:1999}. Synthetic guarded domain
theory also provides the mathematical basis~\cite{bizjak-birkedal:2018} for \textbf{Iris}, a higher-order guarded
separation logic that has been used to develop operationally-based program
logics for sophisticated programming languages involving higher-order store,
concurrency, and a number of other computational effects~\cite{iris:2015,iris:2018}.

\subsection{Multi-clock guarded recursion and coinduction}

Unlike both classical and ordinary synthetic domain theory, the synthetic
guarded domain theory is \emph{effective} in the sense that it gives rise to
type theories satisfying the \emph{canonicity}
property~\cite{gratzer-birkedal:2022}; this means that synthetic guarded domain
theory is a \emph{programming language} in addition to a semantic universe for
denotational semantics. One early application of guarded recursion in this
sense was to provide a more ergonomic and compositional method to write
programs involving coinductive types or \emph{final coalgebras}.

Consider the type of infinite streams $\Con{S}A$ as an example; this
type is the final coalgebra for the endofunctor $\Con{F}_AX = A\times X$, and
because $\Con{F}_A$ is $\omega$-cocontinuous we may compute $\Con{S}A$ as the
limit of the $\omega$-chain $\Con{F}_A^n1 \cong A^n$ by Ad\'amek's theorem. A
stream producer $\Mor[\alpha]{X}{\Con{S}A}$ must therefore decompose into a
cone of finite approximations $\Mor[\alpha_n]{X}{A^n}$ for all $n\in \omega$;
in simpler terms, we must be able to compute any finite approximation of a
stream.
It is not difficult to imagine programming \emph{partial} functions on streams
$\Mor|partial|[\beta]{\Con{S}A}{\Con{S}B}$ by general recursion; such a
programming style is easily supported in languages like Haskell. But what is
the appropriate linguistic construct for defining \emph{total} functions
$\Mor[\beta]{\Con{S}A}{\Con{S}B}$? Just as in the dual case for
\emph{inductive} data, a programming language must verify that recursive calls
are justified and reject any recursive calls that would make (for instance) the
projections $\Mor[\beta_n]{\Con{S}A}{B^n}$ ill-defined.

One method to ensure that recursive functions on coinductive types are total is
to impose a \DefEmph{syntactic guardedness} check: every recursive call must be
wrapped in a call to a constructor. Syntactic guardedness checks are employed
in several type theoretic languages, such as Agda~\cite{norell:2007},
Coq~\cite{coq:reference-manual}, and Idris~\cite{brady:2013,brady:2021}, but
they are unfortunately very brittle and not at all conducive to compositional
and modular programming with higher-order functions. Type-based approaches such
as \DefEmph{sized types} have been proposed as a more compositional alternative
to syntactic checks~\cite{hughes-pareto-sabry:1996}, but the meaning of sized
types as they are used in practice remains poorly understood --- for instance,
the version of sized types implemented in Agda is clearly inconsistent as it
asserts the well-foundedness of an order with $\infty\prec\infty$; yet it
remains unclear whether many sized Agda programs would survive the transition
to a system in which $\infty\not\prec\infty$.\footnote{See
\url{https://github.com/agda/agda/issues/2820} for a discussion of the
inconsistency of Agda's sized types.}

A \emph{sound} and thus more promising type-based approach to ensuring the
guardedness of recursive calls arises from the later modality $\Ltr$, first viewed as a programming construct by
Nakano~\cite{nakano:2000}.\footnote{The calculus of Nakano~\cite{nakano:2000} was subsequently connected to metric domain theory by Birkedal,
Schwinghammer, and St\o{}vring~\cite{birkedal-schwinghammer-stovring:2010}.} The
idea is to approximate the coinductive type $\Con{S}A \cong A \times\Con{S}A$
by the \emph{guarded recursive} type $\Con{S}\Sub{\Ltr}A \cong A \times
\Ltr{\Con{S}\Sub{\Ltr}A}$:
\[
  \Con{S}\Sub{\Ltr} : \Con{Type} \to \Con{Type}\qquad
  \prn{\dblcolon}: A \to \Ltr \Con{S}\Sub{\Ltr} A \to \Con{S}\Sub{\Ltr} A
\]

The guarded fixed point operator then allows recursive definitions of functions
on guarded streams, with the caveat that recursive calls must appear underneath
the later modality. While this semantic / type-based restriction does
automatically ensure totality, it is too conservative: we cannot, for instance,
define the projection functions $\Con{S}\Sub{\Ltr}A\to A^n$. For instance, the
following attempted definition is not well-typed:
\[
  \begin{mathblock}
    \Con{take} : \mathbb{N} \to \Con{S}\Sub{\Ltr} A \to \Con{List} \; A\\
    \Con{take} \; 0 \; u = []\\
    \Con{take} \; (n+1) \; (x \dblcolon u) = x  \dblcolon  {\color{red}\Con{take} \; n \; u}
  \end{mathblock}
\]

\subsubsection{Atkey and McBride's clock-indexed guarded recursion}

At the heart of the problem discussed above is the fact that the guarded
streams only approximate the coinductive streams. The remarkable suggestion of
Atkey and McBride~\cite{atkey-mcbride:2013} is to \emph{define} real
coinductive types in terms of their guarded approximations by adding an
additional notion of \DefEmph{clock} to the language; with this combination of
features, arbitrary functions on coinductive types can be defined using guarded
recursion. In the setting of Atkey and McBride, the later modality ensures that
functions are well-defined and the clocks allow the later modality to be
\emph{removed} in a type-constrained way. The language of Atkey and McBride
contains a new sort of clocks $k$, together with a clock-indexed family of
later modalities $\Ltr_k$ as well as a \DefEmph{clock quantifier} $\forall
k.A\brk{k}$. In the case where $A$ does not depend on the clock variable $k$, a
\DefEmph{clock irrelevance} principle is asserted stating that $A = \forall
k.A$; finally the canonical map $\Mor[\lambda x.\Lambda
k.\Next\prn{x\brk{k}}]{\forall k.A\brk{k}}{\forall k.\Ltr_{k}A\brk{k}}$ is
asserted to have an inverse $\Con{force}$.

A clock can be thought of metaphorically as a ``time stream''; thus an element
of $\forall k. \Ltr_kA$ exhibits an element of $\Ltr_kA$ in all time streams
$k$; thus under this metaphor, the $\Con{force}$ operation simply instantiates
this family at an \emph{earlier} time stream to obtain an element of $A$.  With
the clock-indexed later modality in hand, it is now possible to define
coinductive streams in terms of their guarded approximations by setting
$\Con{S}A \coloneqq \forall k.\Con{S}\Sub{\Ltr_k}A$; thus we may use guarded
recursion to define the $\Con{take}$ function on coinductive streams:
\begin{gather*}
  \begin{mathblock}
    \Con{S}_{\Ltr_k} : \Con{Type} \to \Con{Type}\\
    (\dblcolon) : A \to \Ltr_k \Con{S}_{\Ltr_k} A \to \Con{S}_{\Ltr_k} A\\
  \end{mathblock}
  \qquad
  \begin{mathblock}
    \Con{uncons}_k : \Con{S}_{\Ltr_k} A \to A \times (\Ltr_k \Con{S}_{\Ltr_k} A)\\
    \Con{uncons}_k \; (x \dblcolon u) = (x,u)
  \end{mathblock}
  \qquad
  \begin{mathblock}
    \Con{head} : \Con{S}A \to A\\
    \Con{head} \; u = \Lambda k. \Con{fst}\;\prn{\Con{uncons}_k\; u\brk{k}}
  \end{mathblock}
  \\[8pt]
  \begin{mathblock}
    \Con{tail} : \Con{S} A \to \Con{S}A\\
    \Con{tail} \; u = \Con{force}(\Lambda k. \Con{snd}\;\prn{\Con{uncons}_k\;u\brk{k}})
  \end{mathblock}
  \qquad
  \begin{mathblock}
    \Con{take} : \mathbb{N} \to \Con{S} A \to \Con{List} \; A\\
    \Con{take} \; 0 \; u = []\\
    \Con{take} \; (\Con{suc} \; n) \; u = (\Con{head} \; u)  \dblcolon  (\Con{take} \; n \; (\Con{tail} \; u))
  \end{mathblock}
\end{gather*}

\subsubsection{Bizjak and M\o{}gelberg's clock synchronization; Sterling and Harper's variant}

One aspect of Atkey and McBride's clocks that has proved difficult to account
for in a well-behaved way is that the substitution of clock variables is
restricted to avoid identifying ``time streams'' within types: to be precise,
arbitrary substitutions $\brk{k'/k}$ are permitted in terms, but a substitution
in a \emph{type} is only permitted when it does not cause two distinct clocks
to be identified. To implement this restriction, a somewhat bizarre side
condition on the instantiation rule for $\forall k.A\brk{k}$ is needed, which
unfortunately appears to preclude the generalization of Atkey and McBride's clocks to the
dependently typed setting, \emph{pace} a worthy attempt by
M\o{}gelberg~\cite{mogelberg:2014} which was thwarted by the failure of the
substitution lemma.

Bizjak and M\o{}gelberg~\cite{bizjak-mogelberg:2015} subsequently resolved
these difficulties in 2015 by simply removing the restriction on clock
substitution entirely: to substantiate this simplified language, they construct
a \emph{fibered} presheaf model that supports \emph{diagonal substitutions} of clocks, which
they refer to as \DefEmph{clock synchronization}. With synchronization in
place, there were two remaining problems left unresolved by \opcit:

\begin{enumerate}

  \item the fibered character of the model caused coherence problems that impede the
    interpretation of the syntax of dependent type theory;

  \item also missing was the \DefEmph{clock irrelevance} principle, which
    should at a minimum ensure that the canonical map $\Mor{A}{\forall k.A}$ is
    an isomorphism for any type $A$ that doesn't depend on $k$.

\end{enumerate}

A solution to the coherence problem was found in 2017 by Sterling and
Harper~\cite{sterling-harper:2018:lics}, who employed the Grothendieck
construction to replace the fibered presheaf model by an equivalent
\emph{ordinary} presheaf model. Sterling and Harper also addressed clock
irrelevance in two steps: first, they interpreted the clock quantifier $\forall
k$ as an \emph{intersection type} in an internal realizability
model;\footnote{See also Bizjak and
Birkedal~\cite{bizjak-birkedal:2018:equilogical} for a different realizability
approach to multi-clock guarded recursion via generalized equilogical spaces.}
then they ensured that this intersection is non-trivial by \emph{forcing} the
proposition $\exists k. \top$ in the ambient topos. Around the same time,
Bizjak and M\o{}gelberg~\cite{bizjak-mogelberg:2020} returned to the clock
synchronization model with an analogous solution to the coherence problem and a
more general approach to clock irrelevance: the quantifier
$\forall k$ remains a dependent product, but they restrict the language to types that are
\emph{orthogonal} to the object of clock names. The solution of Bizjak and
M\o{}gelberg is more broadly applicable than that of Sterling and Harper
because intersection need not be meaningful in an arbitrary category whereas
products have a very simple universal property.

\subsection{Goals and structure of this paper}

Most models of single-clock synthetic guarded domain theory are given by
presheaf topoi; the models of multi-clock synthetic guarded domain theory are
also taken in presheaves, but of a different kind than the single-clock
version. Thus the work of
\cite{atkey-mcbride:2013,bizjak-mogelberg:2015,bizjak-mogelberg:2020,sterling-harper:2018:lics}
on multi-clock guarded recurison raises two questions concerning the
relationship between the existing models of single-clock and multi-clock
guarded recursion:

\begin{enumerate}

  \item Does the passage from single-clock to multi-clock topos models have a
    universal property?

  \item Can the multi-clock model be rephrased as a special
    case of the single-clock model of SGDT?

\end{enumerate}

In this paper we answer both questions positively. Each multi-clock topos can
be seen to be a \emph{partial product} or \emph{bagtopos}~\cite{vickers:1992,johnstone:1992,johnstone:1994} for a certain
cocartesian fibration of topoi applied at a given model of single-clock guarded
recursion as hinted by Sterling~\cite[\S2.2.6]{sterling:2021:thesis}; moreover we show that the model of synthetic guarded domain theory
in the multi-clock setting is an instance of the single-clock model generalized to the
\emph{relative Grothendieck topos theory} over a given elementary topos
$\SCat$. Thus we have contributed a completely modular toolkit for negotiating
the two orthogonal axes of variation in multi-clock synthetic guarded domain
theory: the properties of the object of clocks, and the properties of each
later modality $\Ltr\Sub{k}$.

\subsubsection*{Structure of this paper}

In {\cref{sec:topoi}}, we introduce the topos and category theory that
is needed for our technical development. As \emph{relative} topos theory plays
an important role in our work, we
pay special attention to it. In {\cref{sec:sgdt}}, we define
\emph{elementary axioms} for both single-clock and multi-clock synthetic
guarded domain theory in a topos, and we contribute a toolkit for constructing
and transforming models of both.

\begin{enumerate}

  \item In {\cref{sec:stability:psh,sec:stability:sh}} we show that synthetic
    guarded domain theory is stable under both presheaves and localization,
    hence any bounded geometric morphism into a topos model of SGDT lifts this
    model into its domain.

  \item In {\cref{sec:base-models}} we generalize the
    results of Birkedal~\etal~\cite{bmss:2011} by constructing models of SGDT
    in sheaves on frames with a well-founded basis \emph{over an arbitrary base
    topos} with a natural numbers object. This generalization requires a subtle change to the definition of
    well-founded poset, as well as new constructive proofs of the existence of
    term-level guarded fixed points.

\end{enumerate}

In {\cref{sec:single-clock}} we provide an explicit description of the
geometric theories that extant presheaf models of single-clock guarded
recursion classify. In {\cref{sec:multi-clock}} we give a general
construction of a multi-clock model from a single-clock model using the
\emph{bagtopoi} of Vickers~\cite{vickers:1992}: in particular, a multi-clock
topos classifies the theory of a ``bag'' or ``multi-set'' of points of the
corresponding single-clock topos. Our characterization thus provides a
geometrical universal property for multi-clock guarded recursion as a model
construction on topoi, \emph{and} as an explicit transformation of geometric
theories. Finally in {\cref{sec:lifting-to-the-bagtopos}}, we use these
new universal properties to give a new and more abstract proof that the
multi-clock topoi are in fact models of synthetic guarded domain theory in each
clock.

\subsection*{Acknowledgments}

Our thanks to Lars Birkedal for his encouragement and his comments on a draft
version of this paper. This work was supported by a Villum Investigator grant
(no.~25804), Center for Basic Research in Program Verification (CPV), from the
VILLUM Foundation.
\section{Geometric universes, topoi, and logoi}\label{sec:topoi}

Because the topic of the present paper is logico-geometric duality in guarded
domain theory, we are careful to distinguish the spatial aspects of topoi from
the logical ones in both our terminology and notations. Similar conventions are
employed by Vickers~\cite{vickers:1999}, Bunge and Funk~\cite{bunge-funk:2006},
and Anel and Joyal~\cite{anel-joyal:2021}.
\begin{definition}
  A \DefEmph{geometric universe}\footnote{Geometric universes in this sense are
  usually referred to as \emph{elementary topoi with natural numbers objects}.} is a cartesian closed
  category $\ECat$ that has finite limits, a subobject classifier $\Prop$, and a natural
  numbers object $\Nat$.
  A \DefEmph{morphism} $\Mor[F]{\ECat}{\FCat}$ of geometric universes is given by
  a left exact functor $\Mor[\Const{F}]{\ECat}{\FCat}$ equipped with a
  right adjoint $\Const{F}\dashv \Mor[\GSec{F}]{\FCat}{\ECat}$. A \DefEmph{2-morphism}
  $\Mor|2-cell|[\alpha]{F}{G}$ in $\brk{\ECat,\FCat}$ is given by a natural
  transformation $\Mor{\Const{F}}{\Const{G}}$.
\end{definition}

We will write $\GU$ for the meta-2-category of all geometric universes.

\begin{definition}
  A \DefEmph{left exact localization} of a geometric universe $\SCat$ is defined to
  be a morphism $\Mor[L]{\SCat}{\TCat}$ of geometric universes such that the
  right adjoint functor $\Mor[\GSec{L}]{\TCat}{\SCat}$ is fully faithful.
\end{definition}

\begin{definition}
  A \DefEmph{topos} over a geometric universe $\SCat$ is defined to be a geometric
  universe $\Sh{\XTop}$ equipped with a structure morphism
  $\Mor[\XTop]{\SCat}{\Sh{\XTop}}$ of geometric universes such that the gluing
  fibration $\FibMor{\GL{\XTop}=\Sh{\XTop}\downarrow\Const{\XTop}}{\SCat}$ has a small
  separator.  A \DefEmph{morphism} of $\SCat$-topoi $\Mor[f]{\XTop}{\YTop}$ is then a
  morphism in the pseudo-coslice $\SCat\downarrow\GU$,
  \ie a morphism $\Mor[\Sh{f}]{\Sh{\YTop}}{\Sh{\XTop}}$ of geometric universes
  equipped with an isomorphism $\Mor|2-cell|[\Coh{f}]{\XTop}{\YTop;\Sh{f}}$ in
  $\GU\brk{\SCat,\Sh{\XTop}}$ as depicted in the wiring diagram below:
  
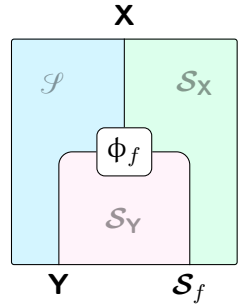
\[
    \begin{tikzpicture}
      \node[string/frame = 3cm and 3cm] (rect) at (0,0) {};
      \node[string/node] (phi) at (rect.center) {$\Coh{f}$};
      \draw[string/wire] (rect.north) -- (phi.north) node [above,at start] {$\XTop$};
      \draw[string/crooked wire,spath/save = bendl] (phi.west) -| (rect.240) node {$\YTop$};
      \draw[string/crooked wire,spath/save = bendr] (phi.east) -| (rect.300) node {$\Sh{f}$};
      \node[string/region label, between = phi.north west and rect.north west] {$\SCat$};
      \node[string/region label, between = phi.north east and rect.north east] {$\Sh{\XTop}$};
      \node[string/region label, between = phi.south and rect.south] {$\Sh{\YTop}$};
      \begin{pgfonlayer}{background}
        \fill[string/region 0] (rect.north) rectangle (rect.south west);
        \fill[string/region 2] (rect.north) rectangle (rect.south east);
        \fill[string/region 1] (spath cs:bendl 0) [spath/use={weld,bendl}] -- (spath cs:bendr 1) [spath/use={weld,reverse,bendr}];
      \end{pgfonlayer}
    \end{tikzpicture}
  \]

  We will write $f^*\dashv f_*$ for the adjunction
  $\Const{\Sh{f}}\dashv\GSec{\Sh{f}}$. A \DefEmph{2-morphism} $\Mor|2-cell|[\alpha]{f}{g}$ in
  $\brk{\XTop,\YTop}$ is defined to be a 2-morphism
  $\Mor|2-cell|[\Sh{\alpha}]{\Sh{f}}{\Sh{g}}$ such that the following wiring diagrams are equal (denote the same 2-cell):
  
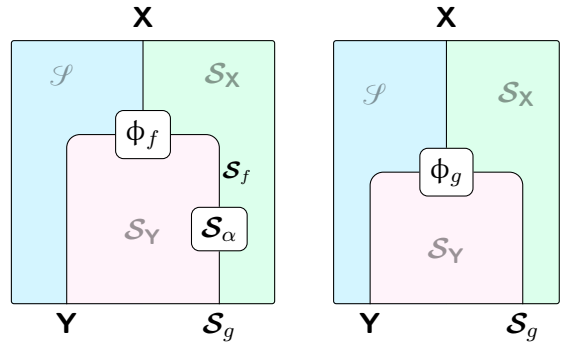
\[
    \begin{tikzpicture}[baseline=(rect.base)]
      \node[string/frame = 3.5cm and 3.5cm] (rect) at (0,0) {};
      \node[string/node] (phi/f) at ([yshift=.5cm]rect.center) {$\Coh{f}$};
      \node[string/node] (alpha) at ([yshift=1cm]rect.300) {$\Sh{\alpha}$};

      \draw[string/wire] (rect.north) -- (phi/f.north) node [above,at start] {$\XTop$};
      \draw[string/crooked wire,spath/save = bendl] (phi/f.west) -| (rect.240) node {$\YTop$};
      \draw[string/crooked wire,spath/save = bendr0] (phi/f.east) -| node [near end, right,inner sep=1pt] {\small$\Sh{f}$} (alpha.north);
      \draw[string/wire,spath/save = bendr1] (alpha.south) -- node {$\Sh{g}$} (rect.300);
      \node[string/region label, between = phi/f.north west and rect.north west] {$\SCat$};
      \node[string/region label, between = phi/f.north east and rect.north east] {$\Sh{\XTop}$};
      \node[string/region label, between = phi/f.south and rect.south] {$\Sh{\YTop}$};
      \begin{pgfonlayer}{background}
        \fill[string/region 0] (rect.north) rectangle (rect.south west);
        \fill[string/region 2] (rect.north) rectangle (rect.south east);
        \fill[string/region 1] (spath cs:bendl 0) [spath/use={weld,bendl}] -- (spath cs:bendr1 1) [spath/use={weld,reverse,bendr0}];
      \end{pgfonlayer}
    \end{tikzpicture}
    \qquad
    \begin{tikzpicture}[baseline=(rect.base)]
      \node[string/frame = 3cm and 3.5cm] (rect) at (0,0) {};
      \node[string/node] (phi) at (rect.center) {$\Coh{g}$};
      \draw[string/wire] (rect.north) -- (phi.north) node [at start,above] {$\XTop$};
      \draw[string/crooked wire,spath/save=bendl] (phi.west) -| (rect.240) node {$\YTop$};
      \draw[string/crooked wire,spath/save=bendr] (phi.east) -| (rect.300) node {$\Sh{g}$};
      \node[string/region label, between = phi.north west and rect.north west] {$\SCat$};
      \node[string/region label, between = phi.north east and rect.north east] {$\Sh{\XTop}$};
      \node[string/region label, between = phi.south and rect.south] {$\Sh{\YTop}$};
      \begin{pgfonlayer}{background}
        \fill[string/region 0] (rect.north) rectangle (rect.south west);
        \fill[string/region 2] (rect.north) rectangle (rect.south east);
        \fill[string/region 1] (spath cs:bendl 0) [spath/use={weld,bendl}] -- (spath cs:bendr 1) [spath/use={weld,reverse,bendr}];
      \end{pgfonlayer}
    \end{tikzpicture}
  \]
\end{definition}

We will write $\TOP{\SCat}$ for the meta-2-category of $\SCat$-topoi for a
geometric universe $\SCat$. We will write $\LOG{\SCat} =
\OpCat{\TOP{\SCat}}$ for the meta-2-category obtained by reversing the
1-cells but not the 2-cells; we refer to an object of $\LOG{\SCat}$ as an
$\SCat$-\emph{logos}. Given a topos $\XTop$, we may think of the dual logos
$\Sh{\XTop}$ as the category of $\SCat$-valued \emph{sheaves} on the space
$\XTop$.
Indeed, the \emph{relative} Giraud theorem states that the logos $\Sh{\XTop}$
can be equivalently presented as a left exact localization of a category of
internal diagrams $\brk{\mathbb{C},\SCat}$ for some internal category
$\mathbb{C}$ in $\SCat$.

\subsection{Geometric theories and classifying topoi}

The notion of \emph{geometric theory} elucidates the relationship between topoi and logoi.

\begin{quote}
  Each geometric theory $\TThy$ determines a \emph{classifying topos}
  $\ClTop{\TThy}$ whose points form the category of $\TThy$-models and
  $\TThy$-model homomorphisms; the dual logos $\Sh{\ClTop{\TThy}}$ is then the
  \emph{classifying category} or category of contexts and substitutions for the
  theory $\TThy$. It is also useful to think of $\Sh{\ClTop{\TThy}}$ as the
  universal extension of the geometric universe $\SCat$ by an indeterminate
  $\TThy$-model in the same way that the polynomial ring $A[x]$ is the free
  extension of a ring $A$ by an indeterminate element.
\end{quote}

There are several possible notions of geometric theory over $\SCat$; for the
sake of this paper, we choose a particularly syntactical one. Informally a
geometric theory over $\SCat$ is given by a collection of sorts $\sigma$ and predicates $\overline{x:\sigma}\mid \phi\prn{\vec{x}}$, together with a collection of sequents
$\overline{x:\sigma}\mid \phi\prn{\vec{x}}\vdash \psi\prn{\vec{x}}$ in which
$\phi,\psi$ are defined using $\brc{\Disj{I},\Conj{n},\exists,\exists!,=}$ where $I$
ranges over an object of $\SCat$. To be precise, all these collections are
parameterized in objects of $\SCat$, so the correct category of geometric theories
arises as a \emph{fibration} $\FibMor{\THY{\SCat}}{\SCat}$ whose fiber at $J\in
\SCat$ is the category of $J$-indexed families of geometric theories with
morphisms given by translations of sorts, operations, and derivable sequents.
In practice, all these notions are conveniently manipulated in the internal
type theory of $\SCat$.

\begin{construction}
  The \DefEmph{free finite product completion} $\FP{\mathbb{C}}$ of any
  internal $\SCat$-category $\mathbb{C}$ can be computed explicitly in the
  internal language of $\SCat$ as follows:
  \begin{enumerate}

    \item an object $\Psi\in \FP{\mathbb{C}}$ is a finite $\SCat$-cardinal\footnote{To be very
    clear, by a finite cardinal we mean an element of the natural numbers object of $\SCat$;
    a morphism of finite cardinals $\Mor{m}{n}$ is a function from $\mathbb{N}\Sub{<m}$ to
    $\mathbb{N}\Sub{<n}$.} $\vrt{\Psi}$ together with a ``type assignment'' $\partial_\Psi \in
    \mathbb{C}\Sup{\vrt{\Psi}}$,

    \item a morphism $\Mor[\psi]{\Phi}{\Psi}$ in $\FP{\mathbb{C}}$ is given by a renaming
      $\Mor[\vrt{\psi}]{\vrt{\Psi}}{\vrt{\Phi}}$ together with a morphism
      $\Mor[\partial_\psi]{\partial_\Phi\Sup{\vrt{\psi}}}{\partial_\Psi}$
      in the product category $\mathbb{C}\Sup{\vrt{\Psi}}$.

  \end{enumerate}

  We will write $\Mor[\gl{-}]{\mathbb{C}}{\FP{\mathbb{C}}}$ for the functor
  sending an object to the corresponding unary product.
\end{construction}

\begin{example}\label{ex:thy-obj}
  The \DefEmph{theory of an object} $\ThyObj$ over $\SCat$ has a single sort $\ClkSort$, no
  operations, and no axioms. To be more formal, the sorts of $\ThyObj$ are
  parameterized by the terminal object $\ObjTerm{\SCat}\in\SCat$ and the
  operations and axioms are parameterized by the initial object
  $\ObjInit{\SCat}\in\SCat$. Letting $\mathbb{O} = \FP{\TrmCat}$ be the free internal $\SCat$-category
  with finite products generated by a single object, then $\Sh{\ClTop{\ThyObj}}$
  is the category of $\SCat$-valued presheaves on $\mathbb{O}$.
\end{example}

\begin{example}
   The \DefEmph{theory of a pointed object} $\ThyElt$ over $\SCat$ has a single sort $\ClkSort$,
   a single constant $\Con{k}: \ClkSort$, and no axioms. Recalling $\mathbb{O}$ from Example~\ref{ex:thy-obj}, then
   $\Sh{\ClTop{\ThyElt}}$ is the category of $\SCat$-valued presheaves on
   $\mathbb{O} \downarrow \gl{*}$.
\end{example}

\subsection{Morphisms of topoi as relative topoi}

Let $\Mor[\gamma]{\ETop}{\BTop}$ be a morphism of $\SCat$-topoi. In other
words, we have morphisms $\Mor[\ETop]{\SCat}{\Sh{\ETop}}$,
$\Mor[\BTop]{\SCat}{\Sh{\BTop}}$, $\Mor[\Sh{\gamma}]{\Sh{\BTop}}{\Sh{\ETop}}$,
and an isomorphism $\Mor|2-cell|[\Coh{\gamma}]{\ETop}{\BTop;\Sh{\gamma}}$
as in the following wiring diagram:
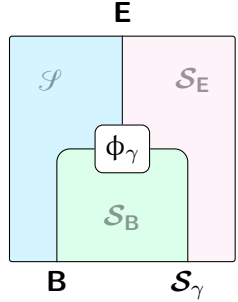
\[
  \begin{tikzpicture}[baseline=(rect.south)]
    \node[string/frame = 3cm and 3cm] (rect) at (0,0) {};
    \node[string/node] (phi) at (rect.center) {$\Coh{\gamma}$};
    \draw[string/wire] (rect.north) -- (phi.north) node [at start, above] {$\ETop$};
    \draw[string/crooked wire,spath/save=bendl] (phi.west) -| (rect.240) node {$\BTop$};
    \draw[string/crooked wire,spath/save=bendr] (phi.east) -| (rect.300) node {$\Sh{\gamma}$};
    \node[string/region label, between = phi.north west and rect.north west] {$\SCat$};
    \node[string/region label, between = phi.north east and rect.north east] {$\Sh{\ETop}$};
    \node[string/region label, between = phi.south and rect.south] {$\Sh{\BTop}$};
    \begin{pgfonlayer}{background}
      \fill[string/region 0] (rect.north west) rectangle (rect.south);
      \fill[string/region 1] (rect.north east) rectangle (rect.south);
      \fill[string/region 2]
      (spath cs:bendl 0) [spath/use = {weld,bendl}]
      -- (spath cs:bendr 1) [spath/use = {reverse,weld,bendr}];
    \end{pgfonlayer}
  \end{tikzpicture}
\]

Forgetting the rest of the structure, the morphism
$\Mor[\Sh{\gamma}]{\Sh{\BTop}}{\Sh{\ETop}}$ of geometric universes \emph{also}
exhibits $\Sh{\ETop}$ as the geometric universe underlying a
$\Sh{\BTop}$-topos~\cite[Lemma~B3.1.10(ii)]{johnstone:2002}. In this scenario,
we shall write $\Mor[\ETop\Sub{\gamma}]{\Sh{\BTop}}{\Sh{\ETop} =
\Sh{\ETop\Sub{\gamma}}}$ for this $\Sh{\BTop}$-topos.
This perspective allows us to take any property $\mathfrak{P}$ of topoi, \eg local
connectedness, and rephrase it as a property of \emph{morphisms} of topoi by
viewing the morphism as a topos over a different base geometric universe:

\begin{convention}[Relative point of view]\label{conv:relative}
  A morphism $\Mor[\gamma]{\ETop}{\BTop}$ of $\SCat$-topoi is said to have
  property $\mathfrak{P}$ if $\ETop\Sub{\gamma}$ has property $\mathfrak{P}$ when viewed
  as a $\Sh{\BTop}$-topos.
\end{convention}

\subsection{Internal presheaves, algebraic topoi, algebraic morphisms}

Let $\mathbb{C}$ be an internal category in a geometric universe $\SCat$; we
may consider the category $\Psh[\SCat]{\mathbb{C}}$ of $\SCat$-valued
presheaves on $\mathbb{C}$. The constant presheaves functor
$\Mor[\PrTop{\mathbb{C}}]{\SCat}{\Psh[\SCat]{\mathbb{C}}}$ is then an
$\SCat$-topos with $\Sh{\PrTop{\mathbb{C}}} = \Psh[\SCat]{\mathbb{C}}$;
following the terminology of Vickers~\cite{vickers:1999} and
Johnstone~\cite{johnstone:1992} we will refer to any $\SCat$-topos equivalent
to one of the form $\PrTop{\mathbb{C}}$ as a \emph{algebraic
$\SCat$-topos}.\footnote{The ``algebraic topos'' terminology tracks an precise
analogy with algebraic dcpos discussed by Vickers~\cite{vickers:1999}.} Thus:

\begin{definition}[Algebraic topos]\label{def:algebraic-topos}
  An $\SCat$-topos $\XTop$ is called \emph{algebraic} when there exists an
  internal category $\mathbb{C}\in\SCat$ and an equivalence of $\SCat$-topoi
  $\Mor{\XTop}{\PrTop{\mathbb{C}}}$.

\end{definition}

Employing \cref{conv:relative} we generalize \cref{def:algebraic-topos} to
morphisms of topoi.\footnote{Not to be confused with the convention of
referring to the inverse image $\Mor[f^*]{\Sh{\YTop}}{\Sh{\XTop}}$ of a
morphism of topoi $\Mor[f]{\XTop}{\YTop}$ as an ``algebraic morphism'', which
we do not employ in this paper.}

\begin{definition}[Algebraic morphism]\label{def:algebraic-morphism}
  A morphism $\Mor[\beta]{\ETop}{\BTop}$
  of $\SCat$-topoi is likewise called \emph{algebraic} when
  $\ETop\Sub{\beta}$ is an algebraic $\Sh{\BTop}$-topos, \ie there exists an
  internal category $\mathbb{C}\in\Sh{\BTop}$ such that there exists an
  equivalence of $\Sh{\BTop}$-topoi $\Mor{\ETop_\beta}{\PrTop{\mathbb{C}}}$.
\end{definition}

In the scenario of \cref{def:algebraic-morphism}, we will speak of
$\Mor[\beta]{\ETop}{\BTop}$ as the algebraic morphism presented by the
$\Sh{\BTop}$-category $\mathbb{C}$.

\begin{observation}[{Johnstone~\cite[Lemma~B2.5.3]{johnstone:2002}}]\label{obs:algebraic-composition}
  If $\Mor{\ETop}{\BTop}$ is the algebraic morphism presented by an
  internal category $\mathbb{E}\in\Sh{\BTop}$ and $\Mor{\FTop}{\ETop}$ is the
  algebraic morphism presented by an internal category
  $\mathbb{O}\in\Sh{\ETop}\simeq \Psh[\Sh{\BTop}]{\mathbb{E}}$, then the
  composite $\Mor{\FTop}{\BTop}$ is the algebraic morphism presented by the
  internal category $\mathbb{E}\rtimes\mathbb{O}\in\Sh{\BTop}$ whose objects
  are given by pairs $\prn{u,v}$ with $u\in\mathbb{E}$ and $v\in\mathbb{O}u$ such that a
  morphism $\Mor{\prn{u,v}}{\prn{u',v'}}$ is given by a pair $\prn{f,g}$ where
  $\Mor[f]{u}{u'}$ and $\Mor[g]{v}{f^*v'}$.
\end{observation}

Thus algebraic morphisms are closed under composition, and moreover,
composition of algebraic morphisms corresponds to the \emph{Grothendieck
construction} for the internal categories that determine them.

\subsubsection{Slices and \'etale morphisms}

Let $\SCat$ be a geometric universe and fix an object $A\in\SCat$. It is
sometimes referred to as the ``fundamental theorem of topos theory'' that the
slice $\SCat\downarrow{}A$ is again a geometric universe; moreover, the pullback functor
$\Mor[A^*]{\SCat}{\SCat\downarrow A}$ can be seen to be the left adjoint part of a
morphism of geometric universes $\Disc{A} = \prn{A^*\dashv A_*}$.  Thus $\Disc{A}$
can be viewed as a \emph{discrete} $\SCat$-topos where $\Sh{\Disc{A}}=\SCat\downarrow {A}$ such
that the points of the $\SCat$-valued topos $\Disc{A}$ are exactly the elements of $A$.
Such a topos is usually referred to as \emph{\'etale}:

\begin{definition}
  An $\SCat$-topos $\ATop$ is called \emph{\'etale} when there exists an object
  $A\in\SCat$ and an equivalence of $\SCat$-topoi $\Mor{\ATop}{\Disc{A}}$.
\end{definition}

Via \cref{conv:relative} we generalize the notion of \'etale $\SCat$-topos to
morphisms between $\SCat$-topoi:

\begin{definition}
  A morphism $\Mor[p]{\ETop}{\BTop}$ of $\SCat$-topoi is called \emph{\'etale}
  when the $\Sh{\BTop}$-topos $\ETop_p$ is \'etale, \ie there exists a sheaf
  $B\in\Sh{\BTop}$ together with an equivalence of $\Sh{\BTop}$-topoi $\Mor{\ETop_p}{\Disc{B}}$.
\end{definition}

As any object $A\in\SCat$ determines a discrete internal $\SCat$-category
$\Con{el}\,A$, it is not difficult to see that any \'etale $\SCat$-topos is
\emph{also} algebraic in a canonical way: we have $\Disc{A} =
\PrTop{\Con{el}\,A}$.  There is furthermore an analogue to
\cref{obs:algebraic-composition} concerning the composition of \'etale
morphisms of topoi:

\begin{observation}
  If $\Mor{\ETop}{\BTop}$ is the \'etale morphism presented by a sheaf
  $E\in\Sh{\BTop}$ and $\Mor{\FTop}{\ETop}$ is the \'etale morphism presented
  by a sheaf $F\in\Sh{\ETop} \simeq \Sh{\BTop}\downarrow E$, then the composite
  $\Mor{\FTop}{\BTop}$ is the \'etale morphism presented by the dependent sum
  $\Sum{E}{F}\in\Sh{\BTop}$.
\end{observation}

\subsection{Partial products of topoi}

We recall the theory of partial products from Johnstone~\cite{johnstone:2002}. Let $\KCat$ be a
finitely complete 2-category and let $\Mor[p]{E}{B}$ be an cocartesian
fibration in $\KCat$ in the sense of
\cite[Definition~B4.4.1]{johnstone:2002}.\footnote{\emph{Op.\ cit.}\ refers to
these as \emph{opfibrations}.} Given an object $A\in \KCat$, a \emph{partial
product cone} over $\prn{p,A}$ at stage $C\in\KCat$ is defined to be a diagram
of the following form in $\KCat$, which we write as a pair $\prn{u,\epsilon}$:
\[
  \begin{tikzpicture}[diagram]
    \SpliceDiagramSquare<sq/>{
      nw = u^*E,
      ne = E,
      sw = C,
      se = B,
      east = p,
      south = u,
      west = u^*p,
      north = q,
      nw/style = pullback,
    }
    \node (A) [left = of sq/nw] {$A$};
    \draw[->] (sq/nw) to node [above] {$\epsilon$} (A);
  \end{tikzpicture}
\]

Johnstone defines a morphism between partial product cones
$\Mor{\prn{u,\epsilon}}{\prn{u',\epsilon'}}$ to be a pair $\prn{\alpha,\beta}$
of 2-cells as depicted below, where $\Mor[\alpha_!]{u^*E}{u'^*E}$ is the
1-cell assigning cocartesian lifts along $\Mor|2-cell|[\alpha]{u}{u'}$:
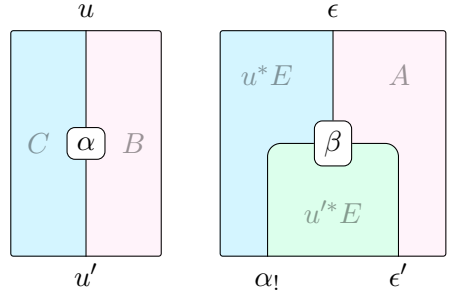
\[
  \begin{tikzpicture}[baseline=(rect.base)]
    \node[string/frame = 2cm and 3cm] (rect) at (0,0) {};
    \node[string/node] (alpha) at (rect.center) {$\alpha$};
    \draw[string/wire] (rect.north) -- (alpha.north) node [above,at start] {$u$};
    \draw[string/wire] (alpha.south) -- (rect.south) node {$u'$};
    \node[string/region label, between = alpha.west and rect.west] {$C$};
    \node[string/region label, between = alpha.east and rect.east] {$B$};
    \begin{pgfonlayer}{background}
      \fill[string/region 0] (rect.north west) rectangle (rect.south);
      \fill[string/region 1] (rect.north east) rectangle (rect.south);
    \end{pgfonlayer}
  \end{tikzpicture}
  \qquad
  \begin{tikzpicture}[baseline=(rect.base)]
    \node[string/frame = 3cm and 3cm,spath/save=rectpath] (rect) at (0,0) {};
    \node[string/node] (beta) at (rect.center) {$\beta$};
    \draw[string/wire] (rect.north) -- (beta.north) node [above,at start] {$\epsilon$};
    \draw[string/crooked wire, spath/save = left bend] (beta.west) -| (rect.240) node {$\alpha_!$};
    \draw[string/crooked wire, spath/save = right bend] (beta.east) -| (rect.300) node {$\epsilon'$};
    \node[string/region label, between = beta.north west and rect.north west] {$u^*E$};
    \node[string/region label, between = beta.north east and rect.north east] {$A$};
    \node[string/region label, between = beta.south and rect.south] {$u'^*E$};
    \begin{pgfonlayer}{background}
      \fill[string/region 0] (rect.north west) rectangle (rect.south);
      \fill[string/region 1] (rect.north east) rectangle (rect.south);
      \fill[string/region 2]
        (spath cs:left bend 0) [spath/use = {weld,left bend}]
        -- (spath cs:right bend 1) [spath/use = {reverse,weld,right bend}];
    \end{pgfonlayer}
  \end{tikzpicture}
\]

We have a \emph{2-fibration}~\cite{buckley:2013} of partial product cones
$\FibMor{\PPCone{p}{A}}{\KCat}$ whose fiber at each $C\in \KCat$ is the
category of partial product cones for $\prn{p,A}$ with vertex $C$ as defined
above. Then the partial product of $\prn{p,A}$ is defined to be a representing
object $\PProd{p}{A}\in\KCat$ for the fibration
$\FibMor{\PPCone{p}{A}}{\KCat}$, if it exists. In other words, for any
$C\in\KCat$ the category of morphisms $\Mor{C}{\PProd{p}{A}}$ is equivalent to
the category of partial product cones for $\prn{p,A}$ with vertex $C$.
Johnstone~\cite{johnstone:2002} points out that the partial product can be written in
the ``internal language'' of the 2-category $\KCat$ as the polynomial expression
$\PProd{p}{A} = \Sum{b:B}\Prod{e:p\brk{b}}A$.

\section{Elementary synthetic guarded domain theory}\label{sec:sgdt}

In this section, we set down \emph{elementary} axioms for synthetic guarded
domain theory in a geometric universe $\SCat$. In
\cref{sec:stability:psh,sec:stability:sh} we study the stability of these
axioms under the two fundamental topos-theoretic constructions: presheaves and
left exact localization. In \cref{sec:base-models} we generalize the results of
Birkedal~\etal to construct a base model from a frame with a well-founded basis
over any base geometric universe.

\subsection{Elementary axioms for synthetic guarded domain theory}

\begin{definition}
  A \DefEmph{later modality structure} on $\SCat$ is given by a left exact
  endofunctor $\Mor[\Ltr]{\SCat}{\SCat}$ called the \emph{later modality}
  together with a natural transformation $\Mor[\Next]{\Idn{\SCat}}{\Ltr}$.
\end{definition}

\begin{definition}\label{def:well-pointed}
  Following the terminology of Kelly~\cite{kelly:1980} we refer to a later
  modality structure as \DefEmph{well-pointed} when the following identity of
  wiring diagrams holds:
  
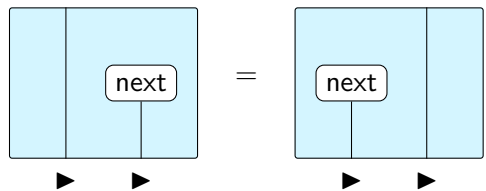
\[
    \begin{tikzpicture}[baseline=(rect.center)]
      \node[string/frame = 2.5cm and 2cm] (rect) at (0,0) {};
      \node[string/node,xshift=.5cm] (next) at (rect.center) {$\Next$};
      \draw[string/wire] (next.south) -- ([xshift=.5cm]rect.south) node {$\Ltr$};
      \draw[string/wire] ([xshift=-.5cm]rect.north) -- ([xshift=-.5cm]rect.south) node {$\Ltr$};
      \begin{pgfonlayer}{background}
        \fill[string/region 0] (rect.north west) rectangle (rect.south east);
      \end{pgfonlayer}
    \end{tikzpicture}
    \quad
    =
    \quad
    \begin{tikzpicture}[baseline=(rect.center)]
      \node[string/frame = 2.5cm and 2cm] (rect) at (0,0) {};
      \node[string/node,xshift=-.5cm] (next) at (rect.center) {$\Next$};
      \draw[string/wire] (next.south) -- ([xshift=-.5cm]rect.south) node {$\Ltr$};
      \draw[string/wire] ([xshift=.5cm]rect.north) -- ([xshift=.5cm]rect.south) node {$\Ltr$};
      \begin{pgfonlayer}{background}
        \fill[string/region 0] (rect.north west) rectangle (rect.south east);
      \end{pgfonlayer}
    \end{tikzpicture}
  \]
\end{definition}

\begin{scholium}
  When the later modality has a left adjoint ${\blacktriangleleft}\dashv\Ltr$,
  the well-pointedness condition of \cref{def:well-pointed} is equivalent to
  the \emph{tick irrelevance} property isolated by Mannaa, M\o{}gelberg, and
  Veltri~\cite{mannaa-mogelberg-veltri:2020}.
\end{scholium}

\begin{remark}
  As the later modality is left exact, it internalizes as a modality
  $\Mor[\Ltr]{\Omega}{\Omega}$ on the subobject classifier that preserves
  finite conjunctions. Likewise, the later modality internalizes as a
  \emph{fibered endofunctor} on the fundamental fibered category of $\SCat$ as
  in Birkedal~\etal~\cite{bmss:2011}, and can hence be used informally in the
  internal dependent type theory of $\SCat$.
\end{remark}

\begin{definition}
  A later modality structure is said to support
  \DefEmph{L\"ob induction} when the sequent $\phi : \Omega\mid
  \Ltr{\phi}\Rightarrow \phi\vdash \phi$ holds in the internal logic of
  $\SCat$.
  It is said to support \DefEmph{guarded recursive
  terms} when for any morphism $\Mor[f]{\Ltr{A}}{A}$ there exists a unique
  element $\Mor[f^\dagger]{\ObjTerm{\SCat}}{A}$ such that $f^\dagger = f^\dagger;\Next\Sub{A};f$.
\end{definition}

The following result follows from the \emph{unique choice} principle valid in any geometric universe.

\begin{restatable}{lemma}{LemGRIffLoeb}\label{lem:gr-iff-loeb}
  A well-pointed later modality structure supports guarded recursive terms if and only if it supports L\"ob induction.
\end{restatable}

\begin{scholium}
  As a consequence of \cref{lem:gr-iff-loeb}, it is rarely necessary to verify
  the guarded recursive terms property, which is always more complex to check
  than L\"ob induction. In fact, \cref{lem:gr-iff-loeb} is an important
  ingredient to our constructivization of the results of
  Birkedal~\etal~\cite{bmss:2011} in \cref{sec:base-models}: it allows us to
  sidestep the very technical and non-constructive proof of Lemma~8.13 in \opcit.
\end{scholium}

We do not here consider algebraic compactness conditions on $\SCat$ with
respect to contractive functors, although these usually play an important role
in the solution of domain equations in synthetic guarded domain
theory~\cite{bmss:2011}.  Instead we take the point of view of Birkedal and
M\o{}gelberg~\cite{birkedal-mogelberg:2013} and advocate solving domain
equations internally to $\SCat$ using term-level guarded recursion on universe
objects. Thus we adopt the following elementary
definition of a model of synthetic guarded domain theory.

\begin{definition}\label{def:elementary-sgdt}
  An \DefEmph{elementary geometric model} of synthetic guarded domain theory is
  given by a geometric universe $\SCat$ equipped with a well-pointed later
  modality $\prn{\Ltr,\Next}$ that supports L\"ob induction.
\end{definition}

Every geometric universe $\SCat$ carries a trivial model of synthetic guarded
domain theory where $\Ltr{A} = \ObjTerm{\SCat}$. It is therefore important to
distinguish non-trivial models in order to implement \emph{adequate}
denotational semantics; the following definition is one possible restraint on
the later modality:

\begin{definition}
  Let $\SCat$ be a geometric universe equipped with a later modality structure
  $\prn{\Ltr,\Next}$, and let $\Mor[S]{\UCat}{\SCat}$ be a morphism of
  geometric universes.  We say that $\prn{\Ltr,\Next}$ is \DefEmph{globally
  adequate} relative to $\Mor[S]{\UCat}{\SCat}$ when the following 2-cell is an
  isomorphism, \ie we have a canonical isomorphism $\Con{force} :
  \GSec{S}\Ltr{\mathbb{N}} \cong \GSec{S}\mathbb{N}$ where $\mathbb{N}$ is the
  natural numbers object of $\SCat$:
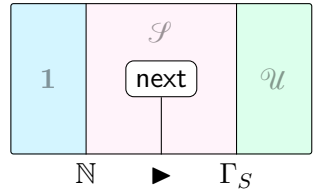
  \[
    \begin{tikzpicture}[baseline = (rect.center)]
      \node[string/frame = 4cm and 2cm] (rect) at (0,0) {};
      \node[string/node] (next) at (rect.center) {$\Next$};
      \draw[string/wire] (next) -- (rect.south) node {$\Ltr$};
      \draw[string/wire,spath/save=nat] ([xshift=-1cm]rect.north) -- ([xshift=-1cm]rect.south) node {$\mathbb{N}$};
      \draw[string/wire,spath/save=gsec] ([xshift=1cm]rect.north) -- ([xshift=1cm]rect.south) node {$\GSec{S}$};
      \node[string/region label, between = spath cs:nat .5 and rect.west] {$\mathbf{1}$};
      \node[string/region label, between = next.north and rect.north] {$\SCat$};
      \node[string/region label, between = spath cs:gsec .5 and rect.east] {$\UCat$};
      \begin{pgfonlayer}{background}
        \fill[string/region 0] (rect.north west) rectangle (spath cs:nat 1);
        \fill[string/region 1] (spath cs:nat 0) rectangle (spath cs:gsec 1);
        \fill[string/region 2] (spath cs:gsec 0) rectangle (rect.south east);
      \end{pgfonlayer}
    \end{tikzpicture}
  \]
\end{definition}

\subsection{Elementary axioms for multi-clock synthetic guarded domain theory}

The multi-clock variants of synthetic guarded domain theory are also
accommodated under \cref{def:elementary-sgdt}; indeed, we may define an
elementary geometric model of \DefEmph{multi-clock synthetic guarded domain
theory} to be a geometric universe $\SCat$ equipped with an object $K\in\SCat$
and an elementary geometric model of synthetic guarded domain theory in the
slice $\SCat\downarrow K$. In this scenario, $K$ is the object of clocks and
clock quantification is implemented by the dependent product functor
$\Mor[K_*]{{\SCat}\downarrow{K}}{\SCat}$.

\subsection{Stability under presheaves}\label{sec:stability:psh}

Let $\mathbb{C}$ be an internal category in a geometric universe $\SCat$, \ie a
small category over $\SCat$. Supposing in addition that $\SCat$ is an
elementary geometric model of synthetic guarded domain theory, we may define a
later modality structure pointwise on the geometric universe
$\Sh{\PrTop{\mathbb{C}}}$ of internal presheaves.  In particular, we define
$\Mor[\Ltr\Sup{\mathbb{C}}]{\Sh{\PrTop{\mathbb{C}}}}{\Sh{\PrTop{\mathbb{C}}}}$ to
take an internal presheaf $E$ to $c\mapsto \Ltr E_c$ in the internal language of
$\SCat$; likewise the natural transformation
$\Mor[\Next\Sup{\mathbb{C}}]{\Idn{\Sh{\PrTop{\mathbb{C}}}}}{\Ltr\Sup{\mathbb{C}}}$ is given
pointwise.
The following \cref{lem:pw-global-adequacy} is verified by rewriting in the
pictorial language of wiring diagrams, using the fact that when $\mathbb{C}$
has a terminal object, the global sections functor
$\Mor[\GSec{\PrTop{\mathbb{C}}}]{\Sh{\PrTop{\mathbb{C}}}}{\SCat}$ is
$\SCat$-cocontinuous and hence preserves the natural numbers object.

\begin{restatable}[Global adequacy in presheaves]{lemma}{LemPshGlobalAdequacy}\label{lem:pw-global-adequacy}
  If the internal category $\mathbb{C}$ has a terminal object and
  $\prn{\Ltr,\Next}$ is globally adequate relative to $\Mor[S]{\UCat}{\SCat}$,
  then $\prn{\Ltr\Sup{\mathbb{C}},\Next\Sup{\mathbb{C}}}$ is globally adequate
  relative to the composite map
  $\Mor[S;\PrTop{\mathbb{C}}]{\UCat}{\Sh{\PrTop{\mathbb{C}}}}$.
\end{restatable}

\begin{theorem}[Stability under presheaves]\label{thm:stability-under-diagrams}
  The pointwise later modality structure
  $\prn{\Ltr\Sup{\mathbb{C}},\Next\Sup{\mathbb{C}}}$ on
  $\Sh{\PrTop{\mathbb{C}}}$ is well-pointed and supports L\"ob induction.  Hence
  the category of diagrams $\Sh{\PrTop{\mathbb{C}}}$ is an elementary geometric
  model of synthetic guarded domain theory.
\end{theorem}

\subsection{Stability under left exact localization}\label{sec:stability:sh}

\begin{construction}[Localized later modality]\label{con:localized-later-modality}
  Let $\Mor[L]{\SCat}{\TCat}$ be a left exact localization, and let $\prn{\Ltr,\Next}$ be a later-modality structure on $\SCat$.
  We may define a canonical later-modality structure $\prn{\Ltr_L,\Next_L}$ on $\TCat$ by conjugating
  with the adjunction $\Const{L}\dashv \GSec{L}$. We define $\Mor[\Ltr_L]{\TCat}{\TCat}$ to be the composite functor
  $\GSec{L};\Ltr;\Const{L}$ and we define $\Mor[\Next_L]{\Idn{\TCat}}{\Ltr_L}$ to
  be the natural transformation depicted in the following wiring diagram, in
  which $\Inv{\epsilon}$ is the inverse to the counit
  $\Mor[\epsilon]{\GSec{L};\Const{L}}{\Idn{\ECat}}$ of the adjunction $\Const{L}\dashv\GSec{L}$.
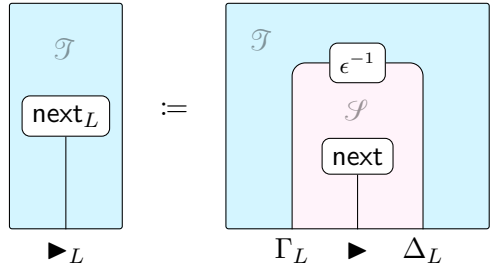
  \[
    \begin{tikzpicture}[baseline = (rect.center)]
      \node[string/frame = 1.5cm and 3cm] (rect) at (0,0) {};
      \node[string/node] (next) at (rect.center) {$\Next_L$};
      \draw[string/wire] (next) -- (rect.south) node {$\Ltr_L$};
      \node[string/region label, between = next.north and rect.north] {$\TCat$};
      \begin{pgfonlayer}{background}
        \fill[string/region 0] (rect.north west) rectangle (rect.south east);
      \end{pgfonlayer}
    \end{tikzpicture}
    \quad
    \coloneqq
    \quad
    \begin{tikzpicture}[baseline = (rect.center)]
      \node[string/frame = 3.5cm and 3cm] (rect) at (0,0) {};
      \node[string/node] (eps) at ([yshift=-.8cm]rect.north) {\small$\Inv{\epsilon}$};
      \node[string/node, between = eps.south and rect.south] (nu) {$\Next$};
      \node[string/region label, below right = .3cm of rect.north west] {$\TCat$};
      \node[string/region label, between = eps.south and nu.north] {$\SCat$};

      \draw[string/crooked wire,spath/save = bend0] (eps.west) -| (rect.240) node {$\GSec{L}$};
      \draw[string/crooked wire,spath/save = bend1] (eps.east) -| (rect.300) node {$\Const{L}$};
      \draw[string/wire] (nu.south) -- (rect.south) node {$\Ltr$} ;
      \begin{pgfonlayer}{background}
        \fill[string/region 0] (rect.north west) rectangle (rect.south east);
        \fill[string/region 1] (spath cs:bend0 0) [spath/use = {weld,bend0}] -- (spath cs:bend1 1) [spath/use = {weld,reverse,bend1}];
      \end{pgfonlayer}
    \end{tikzpicture}
  \]
\end{construction}

The following can be proved pictorially in the language of wiring diagrams; see \cref{appendix:stability} for details.

\begin{restatable}{lemma}{LemLocalizedLaterWellPointed}\label{lem:localized-later-modality-well-pointed}
  If $\prn{\Ltr,\Next}$ is a well-pointed later modality structure on $\SCat$
  and $\Mor[L]{\SCat}{\TCat}$ is a left exact localization, then the later
  modality structure $\prn{\Ltr_L,\Next_L}$ defined in
  \cref{con:localized-later-modality} is well-pointed.
\end{restatable}

\cref{thm:stability-under-localization} follows nearly immediately from an
internal logic argument, using the closure modality associated to any left
exact localization.

\begin{theorem}[Stability under localization]\label{thm:stability-under-localization}
  If the later modality structure $\prn{\Ltr,\Next}$ on $\SCat$ supports L\"ob
  induction, then so does $\prn{\Ltr_L,\Next_L}$ for a left exact localization
  $\Mor[L]{\SCat}{\TCat}$. Hence if $\SCat$ is an elementary geometric model of
  synthetic guarded domain theory, then so is any left exact localization of
  $\SCat$.
\end{theorem}

By the above, we may conclude that \emph{any} category $\ECat$ of
$\SCat$-valued sheaves inherits an elementary model of synthetic guarded domain
theory from the base geometric universe $\SCat$, if it is so-equipped; note
this model could depend on the chosen presentation of $\ECat$ by an $\SCat$-site.
Special care must be taken, as localizations need not preserve global adequacy;
for example, the localization could \emph{trivialize} the later modality in
the sense of making $\Ltr_L{A} = \ObjTerm{\TCat}$ for all $A\in\TCat$.

\subsection{Base models from intuitionistic well-founded posets}\label{sec:base-models}

In the preceeding sections we have shown stability of elementary geometric
models of synthetic guarded domain theory under basic topos theoretic
constructions: presheaves and localization; from these stability properties it
follows that the logos presented by \emph{any} $\ECat$-site inherits guarded
recursion from an elementary geometric model $\ECat$. But how do we construct a
geometric model $\ECat$ in the first place?  Birkedal~\etal~\cite{bmss:2011}
provide a simple recipe for constructing such models, working in the more
restrictive setting of $\SET$-logoi; in particular, it is verified that for any
$\SET$-locale $\XTop$ whose frame of opens $\Opns{\XTop}$ has a well-founded
basis, the logos $\Sh{\XTop}$ is an elementary geometric model of synthetic
guarded domain theory.

In this section, we carry out a (non-trivial) generalization of the results of
\opcit to $\SCat$-locales for an arbitrary geometric universe $\SCat$. Indeed,
not only are the arguments of \opcit non-constructive and thus invalid over an
arbitrary geometric universe: the definition of \emph{well-founded poset} must
be adjusted as well:
\begin{enumerate}

  \item It will not do to define $u<v$ as $\lnot\prn{v\leq u}$ unless $\SCat$
    is boolean; thus the well-founded order on a poset must be an additional
    structure that is compatible with the original order in a certain way.

  \item The classical definition of well-foundedness in terms of infinite descending
    chains is correct if and only if $\SCat$ satisfies dependent choice.

\end{enumerate}

\subsubsection{Basic definitions: intuitionistic well-founded posets and frames}

\begin{definition}
  Let $R\subseteq P\times P$ be a binary relation in $\SCat$. The
  \DefEmph{$R$-accessible elements} of $P$ are the smallest subset
  $\Acc{R}\subseteq P$ spanned by elements $u\in P$ such that for all
  $v\mathrel{R}u$ we have $v \in \Acc{R}$.
  We say that $R$ is \DefEmph{well-founded} when $\Acc{R}\subseteq P$ is $P$ itself.
\end{definition}

\begin{definition}
  Let $\prn{\mathbb{P},\leq}$ be a preorder in $\SCat$; we define a \DefEmph{compatible
  well-founded relation} on $\prn{\mathbb{P},\leq}$ to be a transitive binary subrelation ${\prec} \subseteq
  {\leq}\subseteq \mathbb{P}\times \mathbb{P}$ satisfying the following additional axioms:
  \begin{enumerate}
    \item \emph{Left compatibility.} If $u\leq v$ and $v\prec w$ then $u\prec w$.
    \item \emph{Right compatibility.} If $u \prec v$ and $v\leq w$ then $u\prec w$.
    \item \emph{Well-foundedness.} The relation ${\prec}\subseteq \mathbb{P}\times \mathbb{P}$ is well-founded.
  \end{enumerate}
\end{definition}

\begin{definition}
  We define an \DefEmph{intuitionistic well-founded preorder} to be a triple
  $\prn{\mathbb{P},\leq,\prec}$ in $\SCat$ such that $\prn{\mathbb{P},\leq}$ is a preorder and
  ${\prec}$ is a compatible well-founded relation on $\prn{\mathbb{P},\leq}$. Likewise
  we will speak of an \DefEmph{intuitionistic well-founded poset} to refer to an
  intuitionistic well-founded preorder for which $\leq$ satisfies
  anti-symmetry.
\end{definition}

\begin{definition}[Connected poset]
  Let $\mathbb{P}$ poset object in $\SCat$. We say that $u,v\in\mathbb{P}$ are \emph{comparable} if
  $u\leq v$ or $v \leq u$. A poset object is \emph{connected} if for each $u,v\in\mathbb{P}$ there
  exists a finite sequence $u = c_0,\dots,c_n = v$ in $\mathbb{P}$ such that $c_i$ and $c_{i+1}$
  are comparable for each $i$.
\end{definition}

\begin{definition}[Internal frames]
  A \emph{frame} in $\SCat$ is defined to be a poset object that is closed
  under $\SCat$-joins and finite meets, such that finite meets distribute over
  $\SCat$-joins.
\end{definition}

\begin{definition}[Basis for a frame]
  Let $\mathbb{A}$ be a frame in $\SCat$ and let $\mathbb{K}\subseteq\mathbb{A}$ be a subposet of $A$;
  then $\mathbb{K}$ is called a \emph{basis} for $\mathbb{A}$ when every $u\in \mathbb{A}$ is the least
  upper bound of all the $k\in \mathbb{K}$ such that $k\leq u$.
\end{definition}

\NewDocumentCommand\Pred{}{\Con{p}}

\subsubsection{A base model in sheaves on a frame with well-founded basis}

Any frame $\mathbb{A}$ in $\SCat$ gives rise to an internal site; the underlying
internal category is $\mathbb{A}$ itself, and a family $\brc{{v_i}\leq{u}}$ is covering
when $u=\Disj{i}v_i$. We will write $\widetilde{\mathbb{A}}$ for the $\SCat$-topos
obtained by setting $\Sh{\widetilde{\mathbb{A}}}$ to be the category of $\SCat$-valued
sheaves on the internal site $\mathbb{A}$, letting the structure map
$\Mor[\widetilde{\mathbb{A}}]{\SCat}{\Sh{\widetilde{\mathbb{A}}}}$ take any object of $\SCat$ to
its constant sheaf. Of course we have an embedding of $\SCat$-topoi
$\EmbMor[i]{\widetilde{\mathbb{A}}}{\PrTop{\mathbb{A}}}$ such that the direct image $i_*$ regards
a sheaf as a presheaf and the inverse image $i^*$ sheafifies a presheaf.

We first construct a later modality structure on the presheaf logos
$\Sh{\PrTop{\mathbb{A}}} = \Psh[\SCat]{\mathbb{A}} = \brk{\OpCat{\mathbb{A}},\SCat}$ assuming that there
exists an intuitionistic well-founded preorder $\prn{\mathbb{K},\leq,\prec}$ that forms
a basis for $\mathbb{A}$. For each $u\in \mathbb{A}$, we will write $\mathbb{K}\Sub{\leq u}\subseteq \mathbb{K}$
for the subposet spanned by $k\in \mathbb{K}$ such that $k\leq u$ and $\mathbb{K}\Sub{\prec u}$
for the subposet spanned by $k\in \mathbb{K}$ such that there exists $l\in \mathbb{K}\Sub{\leq
u}$ with $k\prec\Sub{\mathbb{K}} l$.
Following Birkedal~\etal~\cite{bmss:2011} we may define a \emph{predecessor}
operation on $\mathbb{A}$ as a monotone endofunction:
\[
  \begin{mathblock}
    \Mor[\Pred]{\mathbb{A}}{\mathbb{A}}\\
    \Pred\,{u} = \Disj{k\in \mathbb{K}\Sub{\prec u}}{k}
  \end{mathblock}
\]

The predecessor operation induces an essential morphism of $\SCat$-topoi
$\Mor[\PrTop{\Pred}]{\PrTop{\mathbb{A}}}{\PrTop{\mathbb{A}}}$ whose inverse image functor
$\Mor[\PrTop{\Pred}^*]{\Sh{\PrTop{\mathbb{A}}}}{\Sh{\PrTop{\mathbb{A}}}}$ is given by
precomposition with $\Pred$. We have a natural transformation
$\Mor[\nu]{\Idn{\Sh{\PrTop{\Pred}}}}{\PrTop{\Pred}^*}$ defined pointwise by
restriction like so:
\[
  \begin{mathblock}
    \Mor[\nu_E]{E}{\PrTop{\Pred}^*E}\\
    \nu\Sub{E}^u\, e = e\Sub{\vert \Pred{u}}
  \end{mathblock}
\]

The following result follows immediately by computation.

\begin{lemma}
  The pair $\prn{\PrTop{\Pred}^*,\nu}$ comprise a well-pointed later modality structure on $\Sh{\PrTop{\mathbb{A}}}$.
\end{lemma}

Note that the later modality structure on $\Sh{\PrTop{\mathbb{A}}}$ so-defined need not
support L\"ob induction. Using \cref{con:localized-later-modality} we obtain a
later modality structure $\prn{\Ltr,\Next}$ on the sheaf logos
$\Sh{\widetilde{\mathbb{A}}}$ via the localization
$\Mor[\Sh{i}]{\Sh{\PrTop{\mathbb{A}}}}{\Sh{\widetilde{\mathbb{A}}}}$, and this structure remains
well-pointed by \cref{lem:localized-later-modality-well-pointed}. The following
result is proved using the Kripke--Joyal semantics of $\Sh{\widetilde{\mathbb{A}}}$ over
$\SCat$; see \cref{appendix:base-model} for the details.

\begin{restatable}[Base model]{theorem}{ThmBaseModel}\label{thm:base-model}
  The well-pointed later modality structure $\prn{\Ltr,\Next}$ on
  $\Sh{\widetilde{\mathbb{A}}}$ supports L\"ob induction, hence $\Sh{\widetilde{\mathbb{A}}}$ is
  an elementary geometric model of synthetic guarded domain theory.
\end{restatable}
\section{Classifying topoi in single-clock guarded recursion}\label{sec:single-clock}

As Birkedal~\etal~\cite{bmss:2011} have pointed out, presheaves on a
well-founded poset are an instance of the sheaf models of guarded recursion
considered in \cref{sec:base-models}. This is so because presheaves on a
well-founded poset $\mathbb{P}$ are the same as sheaves on the \emph{algebraic locale}
$\PrTop{\mathbb{P}}$, whose frame of opens $\Opns{\PrTop{\mathbb{P}}}$
is the free cocompletion of $\mathbb{P}$ under all $\SCat$-joins and whose poset of points $\brk{\PtTop,\PrTop{\mathbb{P}}}$ is the free \emph{filtered} cocompletion of $\OpCat{\mathbb{P}}$. Thus $\mathbb{P}$ will turn out to be a
(well-founded) basis for $\Opns{\PrTop{\mathbb{P}}}$.

\emph{Explication.} Let $\prn{\mathbb{P},\leq,\prec}$ be an intuitionistic well-founded
poset in $\SCat$; we may define an $\SCat$-locale $\PrTop{\mathbb{P}}$ whose frame of
opens $\Opns{\PrTop{\mathbb{P}}}$ consists of all the downsets of $\mathbb{P}$ ordered by
inclusion, \ie the $\SCat$-poset of $\SCat$-poset homomorphisms
$\brk{\OpCat{\mathbb{P}},\Omega}$.\footnote{The frame $\Opns{\PrTop{\mathbb{P}}}$ was
referred to by Birkedal~\etal~\cite{bmss:2011} as the \emph{ideal completion}
of $\mathbb{P}$.  When $\mathbb{P}$ is a total order, $\Opns{\PrTop{\mathbb{P}}}$ is indeed the
ideal completion of $\mathbb{P}$ but this need not be the case otherwise; on the other
hand, the poset of points of the locale $\PrTop{\mathbb{P}}$ under the specialization
order is indeed the ideal completion of $\OpCat{\mathbb{P}}$ by Diaconescu's
theorem~\cite{diaconescu:1975}.} Then the Yoneda embedding
$\EmbMor{\mathbb{P}}{\Opns{\PrTop{\mathbb{P}}}}$ exhibits $\mathbb{P}$ as a basis for $\Opns{\PrTop{\mathbb{P}}}$,
hence by \cref{thm:base-model} we have an elementary geometric model of
synthetic guarded domain theory in $\Sh{\PrTop{\mathbb{P}}}$. Furthermore, the geometric
universe $\Sh{\PrTop{\mathbb{P}}}$ can be seen to be the category $\Psh[\SCat]{\mathbb{P}}$ of
$\SCat$-valued presheaves on the internal site $\mathbb{P}$.

\begin{lemma}\label{lem:iwfp-adequacy}
  Let $\prn{\mathbb{P},\leq,\prec}$ be an intuitionistic well-founded poset in
  $\SCat$. We will write $\mathbb{P}\Sub{\prec}=\Compr{u}{\exists v.u\prec v}$
  for the subposet spanned by elements lying strictly below another element. If
  both $\mathbb{P}$ and $\mathbb{P}\Sub{\prec}$ are connected, then the later
  modality structure on $\Sh{\PrTop{\mathbb{P}}}$ is globally adequate relative
  to $\Mor[\PrTop{\mathbb{P}}]{\SCat}{\Sh{\PrTop{\mathbb{P}}}}$.
\end{lemma}

\begin{proof}
  We note that $\GSec{\PrTop{\mathbb{P}}}\Ltr{\mathbb{N}}$ may be computed as
  the limit $\Lim{u\in\mathbb{P}}\Lim{v\prec u}\mathbb{N}$, which is also the
  connected limit $\Lim{u\in\mathbb{P}\Sub{\prec}}\mathbb{N}$; hence
  $\GSec{\PrTop{\mathbb{P}}}\Ltr{\mathbb{N}} \cong \mathbb{N} \cong
  \Lim{u\in\mathbb{P}}\mathbb{N} \cong \GSec{\PrTop{\mathbb{P}}}\mathbb{N}$.
\end{proof}

Most models of guarded recursion used in practice are indeed of this kind; in
addition to the global adequacy result (\cref{lem:iwfp-adequacy}) and the
simplicity of working with presheaves, another advantage of the presheaf models
is that they can be characterized as classifying topoi for remarkably simple
and elegant geometric theories.

\begin{lemma}\label{lem:diaconescu:poset}
  For any poset $\prn{\mathbb{P},\leq}$ in $\SCat$, the
  algebraic $\SCat$-topos $\PrTop{\mathbb{P}}$
  classifies the geometric theory of \emph{filters} on $\mathbb{P}$, \ie
  the theory $\ThyFilt{\mathbb{P}}$ axiomatized below:
  \begin{mathpar}
    \inferrule{
      u\in \mathbb{P}
    }{
      \cdot\vdash \gl{u} : \Con{prop}
    }
    \and
    \inferrule{
      u\leq v \in \mathbb{P}
    }{
      \cdot\mid \gl{u}\vdash \gl{v}
    }
    \and
    \inferrule{
    }{
      \cdot\mid\top\vdash \Disj{u\in \mathbb{P}}\gl{u}
    }
    \and
    \inferrule{
      u,v\in \mathbb{P}
    }{
      \cdot\mid \gl{u}\land\gl{v}\vdash\Disj{w\in \mathbb{P}, w \leq u, w\leq v} \gl{w}
    }
  \end{mathpar}
\end{lemma}
\begin{proof}
  By Diaconescu's theorem~\cite{diaconescu:1975}, the topos $\PrTop{\mathbb{P}}$
  classifies the theory of (internally) $\SCat$-valued flat functors on
  $\mathbb{P}$. A flat functor a poset is the same as a filter on that poset.
\end{proof}

\begin{example}[The topos of trees]\label{ex:topos-of-trees}
  The \emph{topos of trees}~\cite{bmss:2011} over a geometric universe $\SCat$
  is defined to be the algebraic topos $\PrTop{\omega}$ where $\omega$ is the
  natural numbers object of $\SCat$ with its usual order; this is the
  ``standard'' model of synthetic guarded domain theory. We recall that $\PrTop{\omega} = \ClTop{\ThyFilt{\omega}}$ from \cref{lem:diaconescu:poset}; because
  $\prn{\omega,\leq}$ is a total
  order in $\SCat$, the downward-directedness axiom for filters can be dropped
  and we see that $\PrTop{\omega}$ classifies models of the following even simpler geometric theory:
  \begin{mathpar}
    \inferrule{
      n : \omega
    }{
      \cdot \mid \gl{n} : \Con{prop}
    }
    \and
    \inferrule{
      m \leq n : \omega
    }{
      \cdot \mid \gl{m} \vdash \gl{n}
    }
    \and
    \inferrule{
    }{
      \cdot\mid \top \vdash \Disj{i\in \omega}\gl{i}
    }
  \end{mathpar}
\end{example}

\begin{example}[Successor ordinals]
  We may consider the successor $\omega^+ = \omega\star \brc{\infty}$ that
  adjoins a terminal element to $\omega$. The algebraic topos $\PrTop{\omega^+}$
  classifies a geometric theory analogous to that of
  \cref{ex:topos-of-trees}, but it can also be seen to classify a simpler
  \emph{cartesian} theory by virtue of the fact that $\omega^+$ has all finite
  meets:
  \begin{mathpar}
    \inferrule{
      \alpha : \omega^+
    }{
      \cdot \mid \gl{\alpha} : \Con{prop}
    }
    \and
    \inferrule{
      \alpha \leq \beta : \omega^+
    }{
      \cdot \mid \gl{\alpha} \vdash \gl{\beta}
    }
    \and
    \inferrule{
    }{
      \cdot\mid \top \vdash \gl{\infty}
    }
  \end{mathpar}
\end{example}

\begin{lemma}\label{lem:wf-preorder-skeleton}
  Let $\prn{\mathbb{P},\sqsubseteq,\sqsubset}$ be an intuitionistic well-founded
  preorder and let $\Mor[q]{\prn{\mathbb{P},\sqsubseteq}}{\prn{\mathbb{P}',\leq}}$ be its poset
  reflection.
  We may define a
  compatible well-founded order $\prec\subseteq\leq$ on $\mathbb{P}'$ by considering the
  image of $\sqsubset$ in $\mathbb{P}'$, \ie
  $u\prec v \Longleftrightarrow \forall x,y. u = qx \to v = qy \to x\sqsubset y$.
\end{lemma}

\begin{example}[Well-founded trees and the plump ordering]
  Let $\Mor[p]{E}{B}$ be a morphism in $\SCat$ and consider the corresponding
  polynomial endofunctor $\Mor[\TyPoly{p}]{\SCat}{\SCat}$ taking $X\in\SCat$ to $\TyPoly{p}X
  = \Sum{b:B}\Prod{e:p\brk{b}}X$. As $\SCat$ supports
  W-types~\cite[Proposition~3.6]{moerdijk-palmgren:2000}, we may form the
  initial algebra $\Mor[\sigma]{\TyPoly{p}\TyW{p}}{\TyW{p}}$ of this polynomial endofunctor.
  Following Fiore, Pitts, and Steenkamp~\cite[Example~5.4]{fiore-pitts-steenkamp:2021} we
  may equip $\TyW{p}$ as an intuitionistic well-founded
  preorder, letting $\prn{\sqsubseteq,\sqsubset}$ be the smallest relations closed under the
  following:
  \[
    \prn{\forall x : p\brk{a}. cx \sqsubset w} \to \sigma\prn{a,c} \sqsubseteq w
    \qquad
    \prn{\exists x : p\brk{a}. w \sqsubseteq cx}\to w \sqsubset \sigma\prn{a,c}
  \]

  The order above is adapted from Taylor~\cite{taylor:1996} and called the
  \emph{plump ordering} on $\TyW{p}$.  By \cref{lem:wf-preorder-skeleton} we have a
  weakly equivalent intuitionistic well-founded poset $\prn{\TyW{p}/\sim,
  \leq,\prec}$, and thus the classifying topos of filters on this poset carries
  a model of SGDT. It is not difficult to see by unrolling definitions that
  \cref{ex:topos-of-trees} is the special case of this construction for the
  family $\Mor[p]{E}{2}$ whose fibers are $\ObjInit{\SCat}$ and
  $\ObjTerm{\SCat}$.
\end{example}

\section{The universal property of multi-clock guarded recursion}\label{sec:multi-clock}

\subsection{Bagtopoi as partial products}

Given a geometric theory $\TThy$ over $\SCat$, there exists a geometric theory
$\ThyBag{\TThy}$ over $\SCat$ of families of $\TThy$-models indexed by an
object of $\SCat$. We will give an explicit description of $\ThyBag{\TThy}$ for
$\TThy$ a propositional geometric theory (\ie a theory with no sorts and only
nullary predicates).

\begin{definition}[Johnstone~\cite{johnstone:1992}]\label{def:bagthy}
  Let $\TThy$ be a propositional geometric theory. Then $\ThyBag{\TThy}$ is the theory
  with a single sort $\ClkSort$ together with:
  \begin{enumerate}
    \item for every proposition symbol $\phi$ of $\TThy$, a predicate $k : \ClkSort\mid \phi\brk{k} : \Con{prop}$;
    \item for every axiom $\cdot\mid\phi\vdash\psi$ of $\TThy$, an axiom $k:\ClkSort\mid \phi\brk{k}\vdash \psi\brk{k}$.
  \end{enumerate}
\end{definition}

\begin{definition}[Johnstone~{\cite[Proposition~B4.4.16]{johnstone:2002}}]
  Let $\ETop$ be an $\SCat$-topos. Let $\FibMor[p]{\ClTop{\ThyElt}}{\ClTop{\ThyObj}}$ be
  the generic \'etale morphism of $\SCat$-topoi projecting the underlying object from a
  pointed object. The \DefEmph{(lower) bagtopos} over $\ETop$ is the partial product
  $\Bag{\ETop} := \PProd{p}{\ETop}$.
\end{definition}

\begin{observation}[Johnstone {\cite[Proposition~B4.4.16]{johnstone:2002}}]
  For a geometric theory $\TThy$, the bagtopos $\Bag{\ClTop{\TThy}}$
  is the classifying $\SCat$-topos for the theory $\ThyBag{\TThy}$, \ie we have
  $\Bag{\ClTop{\TThy}} \simeq \ClTop{\ThyBag{\TThy}}$.
\end{observation}

\subsection{A universal property for Bizjak and M\o{}gelberg's clocks}

Let $\mathbb{K}$ be the $\SCat$-category having as its objects pairs $\prn{U, f
\in \omega^U}$ with $U$ a finite $\SCat$-cardinal, while morphisms $\prn{U, f} \to
\prn{V , g}$ are given by functions $h : U \to V$ such that $h;g \leq f\in
\omega^U$ in the pointwise ordering. The internal category $\mathbb{K}$ above is exactly
the \emph{category of time objects} written $\mathbb{T}$ by Bizjak and
M\o{}gelberg~\cite{bizjak-mogelberg:2020}, relativized over an arbitrary
geometric universe. It was left unmentioned by \opcit that $\mathbb{K}$ is the
free finite coproduct completion of $\OpCat{\omega}$, an observation that sets
the stage for our present results.

We observe that $\OpCat{\mathbb{K}}$ is the free finite product completion $\FP{\omega}$ of $\omega$; thus we define the
\DefEmph{Bizjak--M\o{}gelberg topos} over $\SCat$ to be the algebraic $\SCat$-topos
$\BMTop \coloneqq \PrTop{\FP{\omega}}$ whose total geometric universe is
$\Sh{\BMTop} = \brk{\mathbb{K},\SCat}$.

\begin{lemma}[Johnstone {\cite[Example~B4.4.17]{johnstone:2002}}]
  If $\ETop$ is the algebraic $\SCat$-topos presented by an $\SCat$-category
  $\mathbb{C}$, then the bagtopos $\Bag{\ETop}$ admits an algebraic
  presentation by $\FP{\mathbb{C}}$; in other words, we have
  $\Bag{\PrTop{\mathbb{C}}}\simeq\PrTop{\FP{\mathbb{C}}}$.
\end{lemma}

\begin{corollary}[Universal property]
  The Bizjak--M\o{}gelberg topos $\BMTop$ is equivalent to the bagtopos
  $\Bag{\PrTop{\omega}}$, hence $\BMTop$ classifies the geometric theory
  $\ThyBag{\ThyFilt{\omega}}$. In other words, the category of morphisms of
  topoi $\Mor{\XTop}{\BMTop}$ is exactly the category of pairs $\prn{K,\phi}$
  where $K\in\Sh{\XTop}$ and $\phi\subseteq K\times \omega$ is an $K$-indexed family of filters on
  $\omega$ internal to $\Sh{\XTop}$.
\end{corollary}

\subsection{A universal property for Sterling and Harper's clocks}

Let $\CLK$ be the $\SCat$-category having as its objects pairs $\prn{U, f
\in \omega^U}$ with $U$ a finite, nonzero $\SCat$-cardinal, while morphisms
$\prn{U, f} \to \prn{V , g}$ are given by functions $h : U \to V$ such that
$h;f \leq g\in \omega^U$ in the pointwise ordering. Observe how $\CLK$ is exactly
the category of clocks described by Sterling and Harper~\cite{sterling-harper:2018:lics},
relativized over an arbitrary geometric universe.

\begin{remark}[Image factorization~{\cite[3.2.11--12]{anel-joyal:2021}}]
  Each morphism of topoi $\Mor[f]{\XTop}{\YTop}$ can be factored in a composition of a
  surjection followed by an embedding. If moreover $f$ is \'etale, then the components of
  its factorization are also \'etale.
\end{remark}

\begin{example}\label{ex:thy-inh-obj}
  The \emph{theory of an inhabited object} $\ThyInh$ over $\SCat$ has a single
  sort $\ClkSort$ together with the axiom $\vdash \exists k : \ClkSort. \top$.
  Let $\mathbb{C}$ be an internal $\SCat$-category, and let $\NE{\mathbb{C}}$
  be the full subcategory of $\FP{\mathbb{C}}$ whose objects are nonzero cardinals.
  Then $\Sh{\ClTop{\ThyInh}}$ is the category of $\SCat$-valued presheaves on $\NE{\TrmCat}$.
\end{example}

\begin{definition}
  Let $\TThy$ be a geometric theory. Then $\ThyIBag{\TThy}$ is obtained by extending
  $\ThyBag{\TThy}$ with an additional axiom requiring that the sort $\ClkSort$ is inhabited.
\end{definition}

\begin{definition}
  Let $\FibMor[e]{\ClTop{\ThyElt}}{\ClTop{\ThyInh}}$ be the surjective part of
  the image factorization for the generic \'etale morphism of $\SCat$-topoi.
  Then the \DefEmph{inhabited bagtopos} of an $\SCat$-topos $\ETop$ is the
  partial product $\IBag{\ETop} := \PProd{e}{\ETop}$.
\end{definition}

\begin{observation}
  For a geometric theory $\TThy$, the inhabited bagtopos $\IBag{\ClTop{\TThy}}$
  is the classifying $\SCat$-topos for the theory $\ThyIBag{\TThy}$, \ie we have
  $\IBag{\ClTop{\TThy}} \simeq \ClTop{\ThyIBag{\TThy}}$.
\end{observation}

We observe that $\CLK\cong\NE{\omega}$; thus we define the
\DefEmph{Sterling--Harper topos} over $\SCat$ to be the algebraic $\SCat$-topos
$\SHTop \coloneqq \PrTop{\NE{\omega}}$ whose total geometric universe is
$\Sh{\SHTop} = \Psh[\SCat]{\prn{\NE{\omega}}}$.

\begin{lemma}
  If $\ETop$ is the algebraic $\SCat$-topos presented by an $\SCat$-category
  $\mathbb{C}$, then the inhabited bagtopos $\IBag{\ETop}$ admits an algebraic
  presentation by $\NE{\mathbb{C}}$; in other words, we have
  $\IBag{\PrTop{\mathbb{C}}}\simeq\PrTop{\NE{\mathbb{C}}}$.
\end{lemma}

\begin{corollary}[Universal property]
  The Sterling--Harper topos $\SHTop$ is equivalent to the inhabited bagtopos
  $\IBag{\PrTop{\omega}}$, hence $\SHTop$ classifies the geometric theory
  $\ThyIBag{\ThyFilt{\omega}}$. In other words, the category of morphisms of
  topoi $\Mor{\XTop}{\SHTop}$ is exactly the category of pairs $\prn{K,\phi}$
  where $K\in\Sh{\XTop}$ is inhabited and $\phi\subseteq K\times \omega$ is a
  $K$-indexed family of filters on $\omega$ internal to $\Sh{\XTop}$.
\end{corollary}

\section{Lifting guarded recursion to the bagtopos}\label{sec:lifting-to-the-bagtopos}

In this section, we combine the relative point of view with our general
stability results for models of SGDT to obtain an abstract proof that the
Bizjak--M\o{}gelberg topos $\BMTop$ carries a model of multi-clock guarded
recursion. The results of this section carry over \emph{mutatis mutandis} to
the other variants of the bagtopos considered in the preceeding section.

\begin{observation}
  The universal property of the Bizjak--M\o{}gelberg topos $\BMTop$ as the
  partial product $\PProd{p}{\PrTop{\omega}}$
  determines a
  morphism of $\SCat$-topoi $\Mor[\epsilon]{\Disc{\ClkSort}}{\PrTop{\omega}}$ in the following
  configuration, where $\Mor{\Disc{\ClkSort}}{\BMTop}$ is the \'etale morphism
  corresponding to the generic object:
  \[
    \begin{tikzpicture}[diagram]
      \SpliceDiagramSquare<sq/>{
        nw = \Disc{\ClkSort},
        sw = \BMTop,
        se = \ClTop{\ThyObj},
        ne = \ClTop{\ThyElt},
        south = \ClkSort,
        east = p,
        west = \ClkSort^*p,
        north = \overline{\ClkSort},
        nw/style = pullback,
      }
      \node (omega) [left = of sq/nw] {$\PrTop{\omega}$};
      \draw[->,exists] (sq/nw) to node [above] {$\epsilon$} (omega);
    \end{tikzpicture}
  \]
\end{observation}

Recall that a point of $\PrTop{\omega}$ is an $\omega$-filter and a point of
$\BMTop$ is an indexed family of $\omega$-filters, and moreover a point of
$\Disc{\ClkSort}$ is such an family equipped with a distinguished index; then the
morphism $\Mor[\epsilon]{\Disc{\ClkSort}}{\PrTop{\omega}}$ should be thought of as taking
that distinguished index to the corresponding filter.

\begin{lemma}\label{lem:clk-psh}
  The morphism $\Mor[\epsilon]{\Disc{\ClkSort}}{\PrTop{\omega}}$ is an
  \emph{algebraic} morphism of $\SCat$-topoi that presents
  $\Disc{\ClkSort}$ as the algebraic $\Sh{\PrTop{\omega}}$-topos $\PrTop{\mathbb{C}}$ for an internal
  $\Sh{\PrTop{\omega}}$-category $\mathbb{C}$ with a terminal object.
\end{lemma}

\begin{proof}
  Letting $\PtTop$ be the terminal $\SCat$-topos (so $\Sh{\PtTop}=\SCat$), we
  observe that the composite $\Disc{\ClkSort}\to \PrTop{\omega}\to\PtTop$ is the
  algebraic morphism of $\SCat$-topoi presented by the semidirect product
  $\FP{\omega}\rtimes \ClkSort$ and $\PrTop{\omega}$ is the algebraic $\SCat$-topos
  presented by $\omega$ itself.
  There is a small fibration $\FibMor[\pi]{\FP{\omega}\rtimes \ClkSort}{\omega}$ that
  projects out the ``value'' assigned to the generic clock, in which cartesian
  morphisms are those that leave everything but the value of the generic clock
  unchanged. As this fibration is small, there exists an internal category
  $\mathbb{E}$ in $\Sh{\PrTop{\omega}}$ such that $\FibMor[\pi]{\FP{\omega}\rtimes
  \ClkSort}{\omega}$ is its externalization; explicitly, the fiber $\mathbb{E}\prn{n}$
  for $n\in \omega$ is equivalent to the full subcategory of $\FP{\omega}$
  spanned by objects of the form $\Gamma\times\gl{n}$.
  We finally deduce that $\Mor[\epsilon]{\Disc{\ClkSort}}{\PrTop{\omega}}$ is the
  algebraic morphism of $\SCat$-topoi presented by the
  $\Sh{\PrTop{\omega}}$-category $\mathbb{E}$: the inverse image functor
  $\epsilon^*$ takes the generic $\omega$-filter $\Yo{\omega}{-}$ to the relatively constant
  $\omega$-filter $\Const{\PrTop{\mathbb{E}}}\Yo{\omega}{-} = \pi^*\Yo{\omega}{-}$.
\end{proof}

\begin{corollary}
  The Bizjak--M\o{}gelberg topos $\BMTop$ carries a globally adequate model of multi-clock synthetic guarded domain theory, parameterized in the generic indexing object $\ClkSort\in\Sh{\BMTop}$.
\end{corollary}

\begin{proof}
  The result follows from \cref{lem:clk-psh} via \cref{lem:iwfp-adequacy}.
  Viewing $\Disc{\ClkSort}$ as an algebraic topos over $\PrTop{\omega}$, we
  automatically have a model of synthetic guarded domain theory by
  \cref{thm:stability-under-diagrams}. Global adequacy follows from
  \cref{lem:iwfp-adequacy} via \cref{lem:pw-global-adequacy}.
\end{proof}

\nocite{bahr-grathwohl-mogelberg:2017,veltri-vezzosi:2020,kristensen-mogelberg-vezzosi:2022,mannaa-mogelberg-veltri:2020}

\nocite{johnstone:1992,johnstone:1994,vickers:1992,bunge-funk:1996,bunge-funk:1998,bunge-funk:2006}

\nocite{johnstone:2002}

\bibliographystyle{entics}
\bibliography{references/refs-bibtex}

%\clearpage
%
\appendix

\section{Proofs of results}\label{sec:appendix}

\subsection{Guarded recursive terms \vs L\"ob induction}

\begin{notation}[Internal language]
  When working in the internal language of an elementary geometric model
  $\ECat$ of synthetic guarded domain theory, we will make use of the notations
  of guarded dependent type theory~\cite{bgcmb:2016}. In particular, we employ
  the notation of \emph{delayed substitutions} for the dependent version of the
  later modality defined schematically below for a pair of dependent types
  $\Gamma\vdash A\ \mathit{type}$ and $\Gamma,x:A\vdash Bx\ \mathit{type}$:
  \[
    \begin{tikzpicture}[diagram]
      \SpliceDiagramSquare{
        width = 5cm,
        ne = \Ltr\prn{\Gamma,x:A\vdash Bx},
        se = \Ltr\prn{\Gamma\vdash A},
        sw = {\Gamma\vdash {\color{magenta}\Ltr{A}}},
        nw = {\Gamma,u:\Ltr{A}\vdash {\color{magenta}\Ltr\brk{x\leftarrow u}Bx}},
        east = \Ltr\pi\Sub{Bx},
        west = \pi\Sub{\Ltr\brk{x\leftarrow u}Bx},
        west/style = fibration,
        sw/style = pullback,
        nw/style = pullback,
        west/node/style = upright desc,
      }

      \node (ssw) [below = of sw] {$\Gamma$};
      \node (sse) [below = of se] {$\Ltr{\Gamma}$};

      \node (sww) [left = 6cm of sw] {$\Gamma\vdash A$};
      \node (ssww) [left = 6cm of ssw] {$\Gamma$};
      \node (nww) [left = 6cm of nw] {$\Gamma,x:A\vdash Bx$};
      \draw[->] (nww) to node [above] {$\Next_B\brk{x\leftarrow u}$} (nw);

      \draw[->] (sww) to node [upright desc] {$\Next_A$} (sw);
      \draw[fibration] (sww) to node [left] {$\pi_A$} (ssww);
      \draw[double] (ssww) to (ssw);
      \draw[fibration] (nww) to node [left] {$\pi\Sub{Bx}$} (sww);

      \draw[fibration] (se) to node [right] {$\Ltr{\pi_A}$} (sse);
      \draw[fibration] (sw) to node [upright desc] {$\pi\Sub{\Ltr{A}}$} (ssw);
      \draw[->] (ssw) to node [below] {$\Next_\Gamma$} (sse);
    \end{tikzpicture}
  \]
\end{notation}

\begin{observation}\label{lem:guarded-unique-choice}
  Let $\Mor[\phi]{A}{\Omega}$ be a predicate that holds for at most one element of $A$;
  then $\Ltr{\exists x:A.\phi{x}}$ implies $\exists u:\Ltr{A}.
  \Ltr\brk{x\leftarrow u}\phi{x}$.
\end{observation}

\LemGRIffLoeb*

\begin{proof}
  The only-if direction is immediate. For the converse, we will employ the
  principle of unique choice.  In particular, we consider the predicate
  $x:A\mid \phi\,x$ defined like so:
  \[
    \phi{x}\coloneqq \prn{x = f\,\prn{\Next_A\,x} \land \forall y. y = f\,\prn{\Next_A\,y} \Rightarrow y = x}
  \]

  To show that there exists $x:A$ such that $\phi{x}$ holds, it suffices by
  L\"ob induction to assume $\Ltr\exists x:A.\phi{x}$ and then exhibit $x:A$
  satisfying $\phi{x}$. By \cref{lem:guarded-unique-choice} we may assume that
  there exists some $u:\Ltr{A}$ such that $\Ltr\brk{z\leftarrow{}u}\phi{z}$
  holds. We choose $x \coloneqq fu$ and must check that $\phi\,\prn{fx}$ holds.
  \begin{enumerate}

    \item \emph{Existence.} To check that $fu = f\,\prn{\Next_A\,\prn{fu}}$, we will verify that
      $u = \Next_A\,\prn{fu}$ in $\Ltr{A}$, or equivalently
      $\Ltr\brk{z\leftarrow u}\prn{z = fu}$. By well-pointedness, this is
      the same as $\Ltr\brk{z\leftarrow u}\prn{z = f\,\prn{\Next_A\,z}}$;
      but
      $z=f\,\prn{\Next_A\,z}$ follows from $\phi{z}$, hence our assumption
      $\Ltr\brk{z\leftarrow u}\phi{z}$ suffices.

    \item \emph{Uniqueness.} Next we must check that for all $y:A$ where $y=
      f\,\prn{\Next_A\,y}$, we have $y= fu$; in other words, we must check that
      $f\,\prn{\Next_A\,y} = fu$. By congruence with $f$ it suffices to check
      that $\Next_A\,y = u$ in $\Ltr{A}$, which (as above) is the same as to
      check that $\Ltr\brk{z\leftarrow u}\prn{y = z}$. This again follows from
      our assumption $\Ltr\brk{z\leftarrow u}\phi{z}$.
      \qedhere

  \end{enumerate}
\end{proof}

\subsection{Stability properties for geometric models of SGDT}\label{appendix:stability}

\begin{observation}\label{lem:pw-ltr-gsec}
  If the internal category $\mathbb{C}$ has a terminal object, then the
  pointwise later modality of $\Sh{\PrTop{\mathbb{C}}}$ commutes with global
  sections in the sense that the following canonical 2-cell is an isomorphism:
  
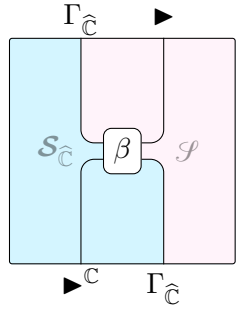
\[
    \begin{tikzpicture}[baseline = (rect.center)]
      \node[string/frame = 3cm and 3cm] (rect) at (0,0) {};
      \node[string/node] (beta) at (rect.center) {$\beta$};
      \draw[string/crooked wire,spath/save = wire0] ([yshift=-.2cm]beta.north west) -| ([xshift=-.55cm]rect.north) node [at end,above] {$\GSec{\PrTop{\mathbb{C}}}$};
      \draw[string/crooked wire,spath/save = wire1] ([yshift=-.2cm]beta.north east) -| ([xshift=.55cm]rect.north) node [at end,above] {$\Ltr$};
      \draw[string/crooked wire,spath/save = wire2] ([yshift=.2cm]beta.south west) -| ([xshift=-.55cm]rect.south) node {$\Ltr\Sup{\mathbb{C}}$};
      \draw[string/crooked wire,spath/save = wire3] ([yshift=.2cm]beta.south east) -| ([xshift=.55cm]rect.south) node {$\GSec{\PrTop{\mathbb{C}}}$};
      \node[string/region label, between = beta.west and rect.west] {$\Sh{\PrTop{\mathbb{C}}}$};
      \node[string/region label, between = beta.east and rect.east] {$\SCat$};
      \begin{pgfonlayer}{background}
        \fill[string/region 0] (rect.north west) rectangle (rect.south east);
        \fill[string/region 1] (spath cs:wire0 0) [spath/use = {weld,wire0}] -- (rect.north east) -- (rect.south east) -- (spath cs:wire3 1) [spath/use={weld,wire3,reverse}] -- (spath cs:wire3 0);
      \end{pgfonlayer}
    \end{tikzpicture}
  \]

  The 2-cell $\beta$ depicted above is the distribution of $\Ltr$ over the limit of a given presheaf.
\end{observation}

\LemPshGlobalAdequacy*

\begin{proof}
  We want to show that the following 2-cell is an isomorphism:
  \begin{equation}\label[diagram]{diag:globpsh1}
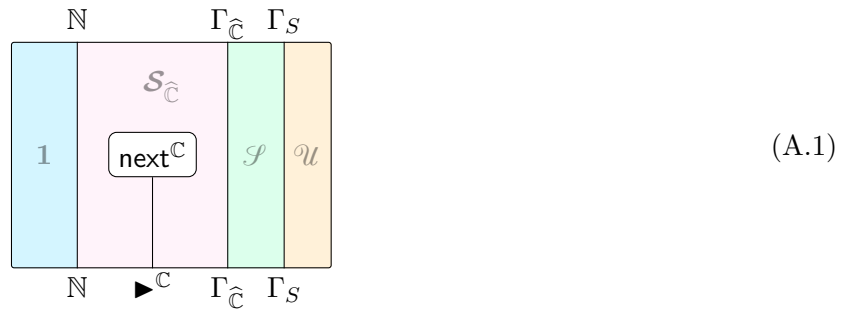
    \begin{tikzpicture}[baseline=(rect.base)]
      \node[string/frame = 4.25cm and 3cm] (rect) at (0,0) {};
      \node[string/node] (nu) at ([xshift=-.25cm,yshift=-1.5cm]rect.north) {$\Next\Sup{\mathbb{C}}$};

      \draw[string/wire, spath/save=wire1] ([xshift=-1.25cm]rect.north) -- ([xshift=-1.25cm]rect.south) node {$\mathbb{N}$} node [at start, above] {$\mathbb{N}$};

      \draw[string/wire, spath/save=wire2] (nu.south) -- ([xshift=-.25cm]rect.south) node {$\Ltr\Sup{\mathbb{C}}$} ;

      \draw[string/wire, spath/save=wire3] ([xshift=.75cm]rect.north) -- ([xshift=.75cm]rect.south) node {$\GSec{\PrTop{\mathbb{C}}}$} node [at start, above] {$\GSec{\PrTop{\mathbb{C}}}$};
      \draw[string/wire, spath/save=wire4] ([xshift=1.5cm]rect.north) -- ([xshift=1.5cm]rect.south) node {$\GSec{S}$} node [at start, above] {$\GSec{S}$};
      \node[string/region label, between = rect.west and spath cs:wire1 .5] {$\mathbf{1}$};
      \node[string/region label, between = nu.north and rect.north] {$\Sh{\PrTop{\mathbb{C}}}$};
      \node[string/region label, between = spath cs:wire3 .5 and spath cs:wire4 .5] {$\SCat$};
      \node[string/region label, between = spath cs:wire4 .5 and rect.east] {$\UCat$};

      \begin{pgfonlayer}{background}
        \fill[string/region 0] (rect.north west) rectangle (spath cs:wire1 1);
        \fill[string/region 1] (spath cs:wire1 0) rectangle (spath cs:wire3 1);
        \fill[string/region 2] (spath cs:wire3 0) rectangle (spath cs:wire4 1);
        \fill[string/region 3] (spath cs:wire4 0) rectangle (rect.south east);
      \end{pgfonlayer}
    \end{tikzpicture}
  \end{equation}

  We may paste an isomorphism onto \cref{diag:globpsh1} and check that the result
  is an isomorphism: the left-hand isomorphism witnesses the preservation of the
  natural numbers object by $\GSec{\PrTop{\mathbb{C}}}$ and the right-hand
  isomorphism is from \cref{lem:pw-ltr-gsec}:
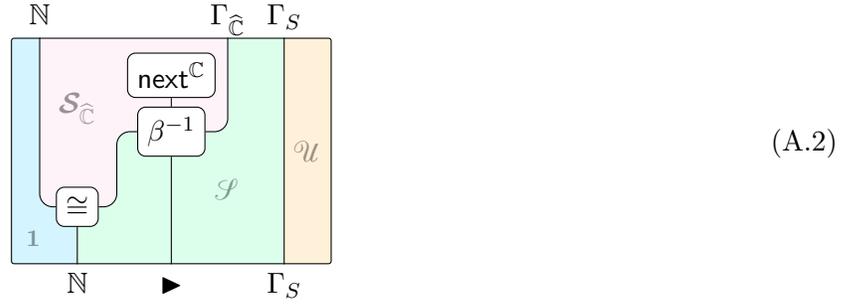
  \begin{equation}\label[diagram]{diag:globpsh2}
    \begin{tikzpicture}[baseline=(rect.base)]
      \node[string/frame = 4.25cm and 3cm] (rect) at (0,0) {};

      \node[string/node] (isol) at ([xshift=-1.25cm,yshift=-2.25cm]rect.north) {$\cong$};
      \node[string/node] (nu) at ([xshift=0cm,yshift=-.5cm]rect.north) {$\Next\Sup{\mathbb{C}}$};
      \node[string/node] (isor) at ([xshift=0,yshift=-1.25cm]rect.north) {$\Inv{\beta}$};

      \draw[string/crooked wire, spath/save=wire1] ([xshift=-1.75cm]rect.north) |- (isol.west) node [at start, above] {$\mathbb{N}$};
      \draw[string/wire,spath/save=wire2] (isol.south) -- ([xshift=-1.25cm]rect.south) node {$\mathbb{N}$};

      \draw[string/wire,spath/save=wire3] (nu.south) -- (isor.north);
      \draw[string/wire,spath/save=wire4] (isor.south) -- ([xshift=0cm]rect.south) node {$\Ltr$};

      \draw[string/crooked wire,spath/save=wire5]  (isol.east) -| ([xshift=-.725cm,yshift=-1.75cm]rect.north) |- (isor.west);

      \draw[string/crooked wire,spath/save=wire6] ([xshift=.75cm]rect.north) |- (isor.east) node [at start, above] {$\GSec{\PrTop{\mathbb{C}}}$};
      \draw[string/wire,spath/save=wire7] ([xshift=1.5cm]rect.north) -- ([xshift=1.5cm]rect.south) node {$\GSec{S}$} node [at start, above] {$\GSec{S}$};

      \node[string/region label, between = isol.south west and rect.south west,yshift=.1cm] {\small$\mathbf{1}$};
      \node[string/region label, above = .7cm of isol.north] {$\Sh{\PrTop{\mathbb{C}}}$};
      \node[string/region label, between = spath cs:wire4 1 and spath cs:wire7 1, yshift=1cm] {$\SCat$};
      \node[string/region label, between = spath cs:wire7 .5 and rect.east] {$\UCat$};

      \begin{pgfonlayer}{background}
        \fill[string/region 0] (rect.north west) -- (rect.south west) -- (spath cs:wire2 1) -- (spath cs:wire2 0) [spath/use={weld,reverse,wire1}]-- (spath cs:wire1 0);
        \fill[string/region 1] (spath cs:wire1 0) [spath/use={weld,wire1}] -- (spath cs:wire1 1) -- (spath cs:wire5 0) [spath/use={weld,wire5}] -- (spath cs:wire5 1) -- (spath cs:wire6 1) [spath/use={weld,wire6,reverse}] -- (spath cs:wire6 0);
        \fill[string/region 2] (spath cs:wire2 1) -- (spath cs:wire7 1) -- (spath cs:wire7 0) -- (spath cs:wire6 0) [spath/use={weld,wire6}] -- (spath cs:wire6 1) -- (spath cs:wire5 1) [spath/use={weld,wire5,reverse}] -- (spath cs:wire5 0) -- (spath cs:wire2 0);
        \fill[string/region 3] (spath cs:wire7 0) rectangle (rect.south east);
      \end{pgfonlayer}
    \end{tikzpicture}
  \end{equation}

  \cref{diag:globpsh2} is equal to the following, which is an isomorphism by assumption:
  \begin{equation}\label[diagram]{diag:globpsh3}
    \begin{tikzpicture}[baseline=(rect.base)]
      \node[string/frame = 4.25cm and 3cm] (rect) at (0,0) {};

      \node[string/node] (isol) at ([xshift=-1.25cm,yshift=-1cm]rect.north) {$\cong$};
      \node[string/node] (nu) at ([xshift=0cm,yshift=-1.5cm]rect.north) {$\Next$};

      \draw[string/crooked wire,spath/save=wire1] ([xshift=-1.75cm]rect.north) |- (isol.west) node [at start, above] {$\mathbb{N}$};
      \draw[string/crooked wire,spath/save=wire2] ([xshift=-.75cm]rect.north) |- (isol.east) node [at start, above] {$\GSec{\PrTop{\mathbb{C}}}$};
      \draw[string/wire,spath/save=wire3] (isol.south) -- ([xshift=-1.25cm]rect.south) node {$\mathbb{N}$};

      \draw[string/wire,spath/save=wire4] (nu.south) -- ([xshift=0cm]rect.south) node {$\Ltr$};

      \draw[string/wire,spath/save=wire5] ([xshift=1.5cm]rect.north) -- ([xshift=1.5cm]rect.south) node {$\GSec{S}$} node [at start, above] {$\GSec{S}$};

      \node[string/region label, between = isol and rect.south west,yshift=.1cm] {$\mathbf{1}$};
      \node[string/region label, above = 0cm of isol.north] {\small$\Sh{\PrTop{\mathbb{C}}}$};
      \node[string/region label, above = .3cm of nu.north] {$\SCat$};
      \node[string/region label, between = spath cs:wire5 .5 and rect.east] {$\UCat$};

      \begin{pgfonlayer}{background}
        \fill[string/region 0] (rect.north west) -- (rect.south west) -- (spath cs:wire3 1) -- (spath cs:wire3 0) -- (spath cs:wire1 1) [spath/use={weld,wire1,reverse}] -- (spath cs:wire1 0);
        \fill[string/region 1] (spath cs:wire1 0) [spath/use={weld,wire1}] -- (spath cs:wire1 1) -- (spath cs:wire2 1) [spath/use={weld,wire2,reverse}] -- (spath cs:wire2 0);
        \fill[string/region 2] (spath cs:wire2 0) [spath/use={weld,wire2}] -- (spath cs:wire2 1) -- (spath cs:wire3 0) -- (spath cs:wire3 1) -- (spath cs:wire5 1) -- (spath cs:wire5 0);
        \fill[string/region 3] (spath cs:wire5 0) rectangle (rect.south east);
      \end{pgfonlayer}
    \end{tikzpicture}
  \end{equation}

  Thus \cref{diag:globpsh1} is an isomorphism.
\end{proof}

\LemLocalizedLaterWellPointed*

\begin{proof}
  We proceed by rewiring in several steps.
  
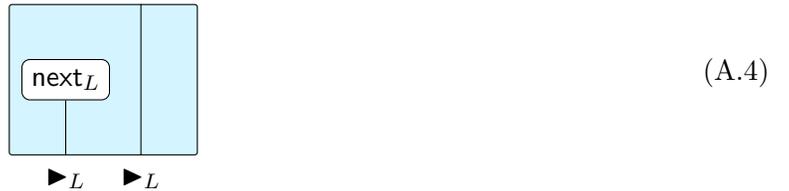
\begin{equation}
    \begin{tikzpicture}[baseline=(rect.base)]
      \node[string/frame = 2.5cm and 2cm] (rect) at (0,0) {};
      \node[string/node,xshift=-.5cm] (next) at (rect.center) {$\Next_L$};
      \draw[string/wire] (next.south) -- ([xshift=-.5cm]rect.south) node {$\Ltr_L$};
      \draw[string/wire] ([xshift=.5cm]rect.north) -- ([xshift=.5cm]rect.south) node {$\Ltr_L$};
      \begin{pgfonlayer}{background}
        \fill[string/region 0] (rect.north west) rectangle (rect.south east);
      \end{pgfonlayer}
    \end{tikzpicture}
  \end{equation}

  First we unfold definitions.
  
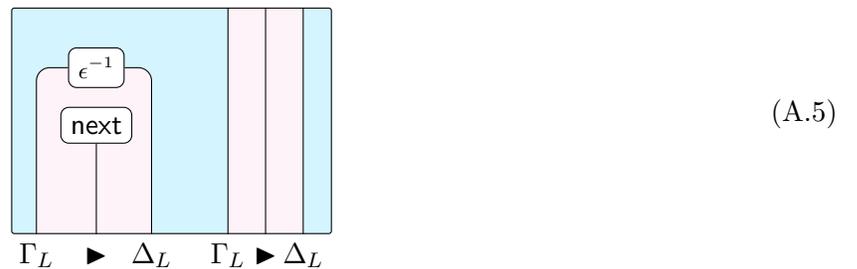
\begin{equation}
    \begin{tikzpicture}[baseline=(rect.base)]
      \node[string/frame = 4.25cm and 3cm] (rect) at (0,0) {};

      \node[string/node] (eps) at ([xshift=-1cm,yshift=-.8cm]rect.north) {\small$\Inv{\epsilon}$};
      \node[string/node, below = .25cm of eps] (nu) {$\Next$};
      \draw[string/crooked wire, spath/save = bendl] (eps.west) -| (rect.220) node {$\GSec{L}$};
      \draw[string/crooked wire, spath/save = bendr] (eps.east) -| (rect.260) node {$\Const{L}$};
      \draw[string/wire] (nu.south) -- ([xshift=-1cm]rect.south) node {$\Ltr$} ;

      \draw[string/wire,spath/save=dir1] ([xshift=.75cm]rect.north) -- ([xshift=.75cm]rect.south) node {$\GSec{L}$};
      \draw[string/wire,spath/save=ltr1] ([xshift=1.25cm]rect.north) -- ([xshift=1.25cm]rect.south) node {$\Ltr$};
      \draw[string/wire,spath/save=inv1] ([xshift=1.75cm]rect.north) -- ([xshift=1.75cm]rect.south) node {$\Const{L}$};
      \begin{pgfonlayer}{background}
        \fill[string/region 0] (rect.north west) rectangle (rect.south east);
        \fill[string/region 1] (spath cs:bendl 0) [spath/use = {weld,bendl}] -- (spath cs:bendr 1) [spath/use = {weld,reverse,bendr}];
        \fill[string/region 1] (spath cs:dir1 0) rectangle (spath cs:inv1 1);
      \end{pgfonlayer}
    \end{tikzpicture}
  \end{equation}

  Next we rewire using the fact that under direct image the inverse to the counit becomes the unit.
  
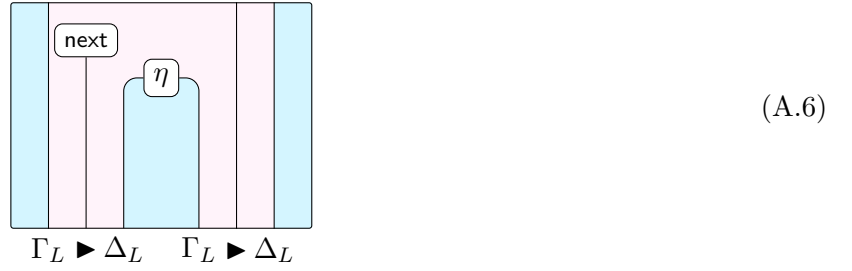
\begin{equation}
    \begin{tikzpicture}[baseline=(rect.base)]
      \node[string/frame = 4cm and 3cm] (rect) at (0,0) {};
      \node[string/node] (nu) at (-1cm,1cm) {\small$\Next$};
      \node[string/node] (eta) at (0,.5cm) {$\eta$};
      \draw[string/crooked wire,spath/save=bendl] (eta.west) -| (-.5cm,-1.5cm) node {$\Const{L}$};
      \draw[string/crooked wire,spath/save=bendr] (eta.east) -| (.5cm,-1.5cm) node {$\GSec{L}$};
      \draw[string/wire,spath/save=dir0] ([xshift=-1.5cm]rect.north) -- ([xshift=-1.5cm]rect.south) node {$\GSec{L}$};
      \draw[string/wire] (nu.south) -- (-1cm,-1.5cm) node {$\Ltr$};
      \draw[string/wire,spath/save=ltr1] ([xshift=1cm]rect.north) -- ([xshift=1cm]rect.south) node {$\Ltr$};
      \draw[string/wire,spath/save=inv1] ([xshift=1.5cm]rect.north) -- ([xshift=1.5cm]rect.south) node {$\Const{L}$};
      \begin{pgfonlayer}{background}
        \fill[string/region 0] (rect.north west) rectangle (spath cs:dir0 1);
        \fill[string/region 0] (spath cs:inv1 0) rectangle (rect.south east);
        \fill[string/region 1] (spath cs:dir0 0) rectangle (spath cs:inv1 1);
        \fill[string/region 0] (spath cs:bendl 0) [spath/use = {weld,bendl}] -- (spath cs:bendr 1) [spath/use = {reverse,weld,bendr}];
      \end{pgfonlayer}
    \end{tikzpicture}
  \end{equation}

  Then we rewire using our assumption that $\prn{\Ltr,\Next}$ is well-pointed.
  
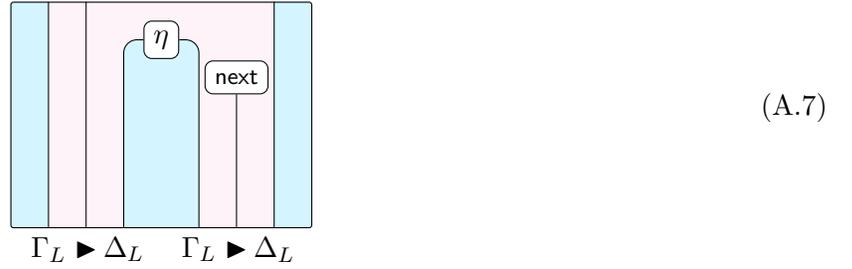
\begin{equation}
    \begin{tikzpicture}[baseline=(rect.base)]
      \node[string/frame = 4cm and 3cm] (rect) at (0,0) {};
      \node[string/node] (nu) at (1cm,.5cm) {\small$\Next$};
      \node[string/node] (eta) at (0,1cm) {$\eta$};
      \draw[string/crooked wire,spath/save=inv0] (eta.west) -| (-.5cm,-1.5cm) node {$\Const{L}$};
      \draw[string/crooked wire,spath/save=dir1] (eta.east) -| (.5cm,-1.5cm) node {$\GSec{L}$};
      \draw[string/wire,spath/save=dir0] ([xshift=-1.5cm]rect.north) -- ([xshift=-1.5cm]rect.south) node {$\GSec{L}$};
      \draw[string/wire,spath/save=ltr0] (nu.south) -- (1cm,-1.5cm) node {$\Ltr$};
      \draw[string/wire,spath/save=ltr1] ([xshift=-1cm]rect.north) -- ([xshift=-1cm]rect.south) node {$\Ltr$};
      \draw[string/wire,spath/save=inv1] ([xshift=1.5cm]rect.north) -- ([xshift=1.5cm]rect.south) node {$\Const{L}$};
      \begin{pgfonlayer}{background}
        \fill[string/region 0] (rect.north west) rectangle (spath cs:dir0 1);
        \fill[string/region 1] (spath cs:dir0 0) rectangle (spath cs:inv1 1);
        \fill[string/region 0] (spath cs:inv1 0) rectangle (rect.south east);
        \fill[string/region 0] (spath cs:inv0 1) [spath/use={weld,reverse,inv0}] -- (spath cs:dir1 0) [spath/use={weld,dir1}] (spath cs:dir1 1);
      \end{pgfonlayer}
    \end{tikzpicture}
  \end{equation}

  We rewire using the fact that the unit becomes the inverse to the counit under inverse image.
  
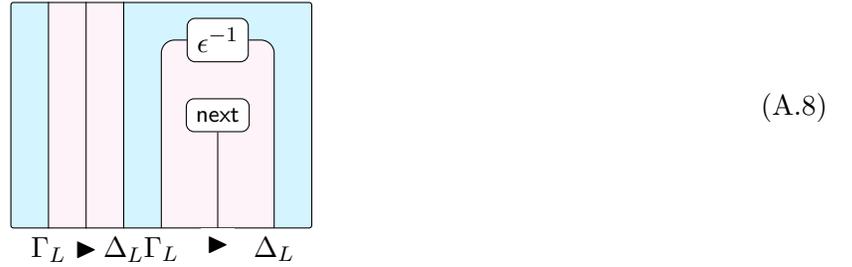
\begin{equation}
    \begin{tikzpicture}[baseline=(rect.base)]
      \node[string/frame = 4cm and 3cm] (rect) at (0,0) {};
      \node[string/node] (nu) at (.75cm,0cm) {\small$\Next$};
      \node[string/node] (eps) at (.75cm,1cm) {$\Inv{\epsilon}$};
      \draw[string/wire,spath/save=ltr1] (nu.south) -- (.75cm,-1.5cm) node [upright desc] {$\Ltr$};
      \draw[string/wire,spath/save=ltr0] ([xshift=-1cm]rect.north) -- ([xshift=-1cm]rect.south) node {$\Ltr$};
      \draw[string/wire,spath/save=inv0] ([xshift=-.5cm]rect.north) -- ([xshift=-.5cm]rect.south) node {$\Const{L}$};
      \draw[string/wire,spath/save=dir0] ([xshift=-1.5cm]rect.north) -- ([xshift=-1.5cm]rect.south) node {$\GSec{L}$};
      \draw[string/crooked wire,spath/save=dir1] (eps.west) -| (0cm,-1.5cm) node {$\GSec{L}$};
      \draw[string/crooked wire,spath/save=inv1] (eps.east) -| (1.5cm,-1.5cm) node {$\Const{L}$};

      \begin{pgfonlayer}{background}
        \fill[string/region 0] (rect.north west) rectangle (spath cs:dir0 1);
        \fill[string/region 1] (spath cs:dir0 0) rectangle (spath cs:inv0 1);
        \fill[string/region 0] (spath cs:inv0 0) rectangle (rect.south east);
        \fill[string/region 1] (spath cs:dir1 1) [spath/use={weld,reverse,dir1}] -- (spath cs:inv1 0) [spath/use={weld,inv1}] (spath cs:inv1 1);
      \end{pgfonlayer}
    \end{tikzpicture}
  \end{equation}

  Folding definitions, we are done.
  
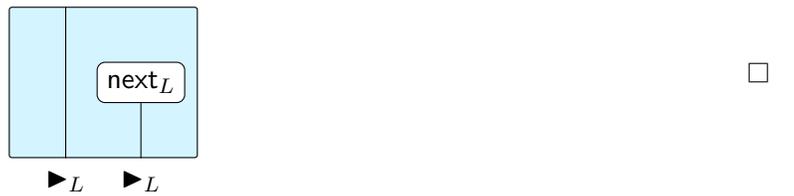
\begin{equation*}
    \begin{tikzpicture}[baseline=(rect.base)]
      \node[string/frame = 2.5cm and 2cm] (rect) at (0,0) {};
      \node[string/node,xshift=.5cm] (next) at (rect.center) {$\Next_L$};
      \draw[string/wire] (next.south) -- ([xshift=.5cm]rect.south) node {$\Ltr_L$};
      \draw[string/wire] ([xshift=-.5cm]rect.north) -- ([xshift=-.5cm]rect.south) node {$\Ltr_L$};
      \begin{pgfonlayer}{background}
        \fill[string/region 0] (rect.north west) rectangle (rect.south east);
      \end{pgfonlayer}
    \end{tikzpicture}
    \qedhere
  \end{equation*}
\end{proof}

\subsection{A base geometric model of SGDT}\label{appendix:base-model}

In this section, recall that $\mathbb{A}$ is a frame in a geometric universe $\SCat$
equipped with a well-founded basis $\mathbb{K}$. We will write
$\Mor[i]{\widetilde{\mathbb{A}}}{\PrTop{\mathbb{A}}}$ for the corresponding embedding of
$\SCat$-topoi. We have obtained a well-pointed later modality structure
$\prn{\Ltr,\Next}$ on $\Sh{\widetilde{\mathbb{A}}}$ from the well-pointed later modality
structure $\prn{\PrTop{\Pred}^*,\nu}$ on $\Sh{\PrTop{\mathbb{A}}}$ using \cref{con:localized-later-modality}, but it remains to show
that the former supports L\"ob induction.

\ThmBaseModel*

\begin{proof}
  This is easily verified in the Kripke-Joyal semantics of $\Sh{\widetilde{\mathbb{A}}}$. Fixing
  $u\in \mathbb{A}$ and a closed sieve $\phi\in\Omega{u}$ such that $u\Vdash
  \prn{\Ltr{\phi}\Rightarrow\phi}=\top$ we must check that $u\Vdash \phi=\top$.
  As $\mathbb{K}$ is a basis for $\mathbb{A}$, we know that $u = \Disj{k\in \mathbb{K}\Sub{u}}k$, so by local
  character it suffices to verify that $k\Vdash k^*\phi = \top$ for each $k\in
  \mathbb{K}\Sub{u}$.
  By well-founded induction we may assume that $l\Vdash l^*\phi = \top$ for all
  $l\prec k$. From this assumption we have $k\Vdash \PrTop{\Pred}^*\prn{k^*\phi
  = \top}$ and hence $k\Vdash \Ltr\prn{k^*\phi=\top}$, so by our assumption
  that $u\Vdash\prn{\Ltr{\phi}\Rightarrow\phi}=\top$ we are done.
\end{proof}

\end{document}